\documentclass[12pt]{article}
\usepackage{enumerate}
\usepackage{amsmath}
\usepackage{amsthm}
\usepackage{amsfonts}
\usepackage{amssymb}
\usepackage{xcolor}
\usepackage[numbers]{natbib}
\usepackage{graphicx,xcolor}
\graphicspath{ {Diagrams/} }

\theoremstyle{plain}
\newtheorem{question}{Question}
\newtheorem{case}{Case}
\newtheorem{problem}[question]{Problem}
\newtheorem{conjecture}[question]{Conjecture}
\newtheorem{theorem}[question]{Theorem}
\newtheorem{proposition}[question]{Proposition}

\newtheorem{lemma}[question]{Lemma}
\newtheorem{remark}[question]{Remark}
\theoremstyle{definition}
\newtheorem{definition}[question]{Definition}

\numberwithin{question}{section}
\numberwithin{equation}{section}
\newtheorem*{theorem*}{Theorem}

\title{Maker-Breaker Percolation Games I: Crossing Grids}
\author{A. Nicholas Day\thanks{Institutionen f\"or matematik och matematisk statistik, Ume{\aa} Universitet, 901 87 Ume{\aa}, Sweden. Emails:  \texttt{a.nick.day@gmail.com} and \texttt{victor.falgas-ravry@umu.se}. Research supported by Swedish Research Council grant 2016-03488.} \and Victor Falgas--Ravry\footnotemark[1]}
\begin{document}
\maketitle
\begin{abstract}
Motivated by problems in percolation theory, we study the following 2-player positional game.  Let $\Lambda_{m \times n}$ be a rectangular grid-graph with $m$ vertices in each row and $n$ vertices in each column. Two players, Maker and Breaker, play in alternating turns. On each of her turns, Maker claims $p$ (as-yet unclaimed) edges of the board $\Lambda_{m \times n}$, while on each of his turns Breaker claims $q$ (as-yet unclaimed) edges of the board and destroys them. Maker wins the game if she manages to claim all the edges of a crossing path joining the left-hand side of the board to its right-hand side, otherwise Breaker wins. We call this game the $(p,q)$-crossing game on $\Lambda_{m \times n}$.

Given $m,n\in \mathbb{N}$, for which pairs $(p,q)$ does Maker have a winning strategy for the $(p,q)$-crossing game on $\Lambda_{m \times n}$? The $(1,1)$-case corresponds exactly to the popular game of Bridg-it, which is well understood due to it being a special case of the older Shannon switching game. In this paper, we study the general $(p,q)$-case. Our main result is to establish the following transition:
 \begin{itemize}
 	\item if $p\geqslant 2q$, then Maker wins the game on arbitrarily long versions of the narrowest board possible, i.e. Maker has a winning strategy for the $(2q, q)$-crossing game on $\Lambda_{m \times(q+1)}$ for any $m\in \mathbb{N}$;
 	 	\item if $p\leqslant 2q-1$, then for every width $n$ of the board, Breaker has a winning strategy for the $(p,q)$-crossing game on $\Lambda_{m \times n}$ for all sufficiently large board-lengths $m$.
 \end{itemize}
Our winning strategies in both cases adapt more generally to other grids and crossing games. In addition we pose many new questions and problems.
\end{abstract}

\section{Introduction}\label{section: introduction}
\subsection{Results and organisation of the paper}\label{subsection: results}
Biased Maker--Breaker games are a central area of research on positional games, in particular due to their intriguing and deep connections to resilience phenomena  in discrete random structures. Much of the research on Maker--Breaker games has focussed on the case where the ``board'' is a complete hypergraph, or an arithmetically-defined hypergraph corresponding to all the solutions to a system of equations in some finite integer interval. Typically the ``winning sets'' that Maker seeks to claim in these games all have the same size.

In this paper we focus on boards and winning sets with rather different properties: we consider rectangular grid graphs, and our winning sets consist of \emph{crossing paths}, whose sizes can vary wildly.

Explicitly, we define the  $(p,q)$-crossing game as follows. Let $\Lambda_{m \times n}$ be the rectangular grid-graph with $m$ vertices in each row and $n$ vertices in each column - our convention is to call $m$ the \emph{length} and $n$ the \emph{width} of the board.  Two players, Maker and Breaker, play in alternating turns, with Maker playing first. On each of her turns, Maker claims $p$ (as-yet unclaimed) edges of the board $\Lambda_{m \times n}$, while on each of his turns Breaker claims $q$ (as-yet unclaimed) edges of the board and destroys them.  The game ends if either Maker manages to claim all the edges of a crossing path joining the left-hand side of the board to its right-hand side, in which case we declare her the winner, or if the board reaches a state where it is not longer possible for Maker to ever claim such a left-right crossing path, in which case we declare Breaker the winner.  A natural question to ask is, given positive integers $m,n,p,q$, which player has a winning strategy for the $(p,q)$-crossing game on $\Lambda_{m \times n}$?

Our main result is the following two theorems, which show that the game undergoes a sharp transition at $p=2q$:

\begin{theorem}\label{theorem: (2q,q)-game}
Let $m,n,p$ and $q$ be natural numbers.  If $p \geqslant 2q$ and $n \geqslant q+1$, then  Maker has a winning strategy for the $(p,q)$-crossing game on $\Lambda_{m  \times n}$.
\end{theorem} 
\begin{theorem}\label{theorem: (2q-1,q)-game}
Let $m,n,p$ and $q$ be natural numbers.  There exists a natural number $m_{0} = m_{0}(n,q)$ such that if $p \leqslant 2q-1$ and $m \geqslant m_{0}$, then  Breaker has a winning strategy for the $(p,q)$-crossing game on $\Lambda_{m  \times n}$.
\end{theorem}
In other words if Maker has at least twice the power of Breaker and the board is wide enough that Breaker cannot win in a single turn, then Maker wins the game no matter how long the board is.  On the other hand, if Maker has strictly less than twice Breaker's power, then Breaker has a winning strategy on all boards that are sufficiently long (with respect to the board width $n$  and Breaker's power $q$).  The proofs of Theorems~\ref{theorem: (2q,q)-game} and~\ref{theorem: (2q-1,q)-game} can be found in Sections~\ref{section: (2q,q)-game} and~\ref{section: (2q-1,q)-game} respectively. As we remark in Section~\ref{section: other graphs and other games}, our strategies for these two games adapt to a number of other games and grids; see in particular Theorem~\ref{theorem: general strip theorem} for a generalisation of Theorem~\ref{theorem: (2q-1,q)-game}.

The rest of this paper is organised as follows: in Section~\ref{subsection: background and motivation} we give some background and motivation for our problem. In Section~\ref{section: preliminaries} we go over some basic definitions and prove some elementary results on crossing games. For completeness, we also record a winning strategy for the $(1,1)$-crossing game on $\Lambda_{(n+1) \times n}$ which allows Maker to play any edge on her first move (this might be folklore --- that Maker has a winning strategy is well-known, but we could not find a reference to the fact any first move will do).
We end this paper in Section~\ref{section: concluding remarks} with a number of questions and open problems, including a discussion of connections to the study of fugacity in statistical physics and some enumeration problems in analytic combinatorics.

\subsection{Background and motivation}\label{subsection: background and motivation}

Maker--Breaker games are a class of positional games which have attracted considerable attention from researchers in combinatorics and discrete probability. The set-up is simple: we have a finite board (a set) $X$, and a collection $\mathcal{W}$ of subsets of $X$ called \emph{winning sets}. Two players, Maker and Breaker, take turns to claim as-yet unclaimed elements of $X$. Maker (typically) plays first, and claims $a$ elements in each of her turns, while Breaker claims $b$ elements on each of his. Maker's aim is to claim all the elements of a winning set $W\in \mathcal{W}$, while Breaker's aim is to thwart her, i.e. to claim at least one element from each winning set. Since the board is finite, no draws are allowed, and the main question is to determine who has a winning strategy.

Maker--Breaker games on graphs have been extensively studied since an influential paper of Chv\'atal and Erd{\H o}s~\cite{ChvatalErdos78} in the late 1970s. Important examples of such games include the connectivity game, the $k$-clique game and the Hamiltonicity game, where the board $X$ consists of the edges of a complete graph on $n$ vertices and the winning sets are spanning trees, $k$-cliques and Hamiltonian cycles respectively.

In their paper Chv\'atal and Erd{\H o}s proved that, for a variety of such games, if $n$ is sufficiently large, then Maker has a winning strategy in the case where $a=b=1$. In each case they then asked how large a bias $b=b(n)$ was required for the $(1,b)$ versions of these games to turn into Breaker's win and provided a surprising and influential \emph{random graph heuristic} for determining the value of these \emph{threshold biases}. Namely, according to this heuristic the threshold bias $b_{\star}$ at which Breaker has a winning strategy should lie  close to the threshold $b$ for a set of $\frac{1}{b+1}\binom{n}{2}$ edges chosen uniformly at random to fail, with high probability, to contain any winning set.

This random graph heuristic has been widely investigated by a large number of researchers, in particular by Beck~\cite{Beck82,Beck85, Beck93, Beck94} and Bednarska and {\L}uczak~\cite{BednarskaLuczak00, BednarskaLuczak01}. Its correctness has been rigorously established for some games, such as the connectivity~\cite{GebauerSzabo09}, $k$-clique~\cite{Beck08} and Hamiltonicity~\cite{Krivelevich11} games, but it has also been shown to fail for other games such as general $H$-games~\cite{BednarskaLuczak00} (where the winning sets are copies of some fixed, finite graph $H$ containing at least three non-isolated vertices).

In a different direction, Stojakovi{\'c} and Szab{\'o}~\cite{StojakovicSzabo05} considered playing these Maker--Breaker games on random boards, by having $X$ consist of the edges of an Erd{\H o}s--R\'enyi random graph $G_{n,p}$. As having fewer edges cannot help Maker, the natural question in this setting is: what is the threshold $p_{\star}$ such that if $p\gg p_{\star}$, then with probability $1-o(1)$ Maker has a winning strategy for the $(1,1)$-crossing game on $G_{n,p}$, while if $p\ll p_{\star}$, then with probability $1-o(1)$ Breaker has a winning strategy.  Stojakovi{\'c} and Szab{\'o} showed that for some games, such as the connectivity games,  $1/b_{\star}$ and $p_{\star}$ are of the same order, but that for others, such as the triangle game, no such relationship holds.

The intriguing connections between Maker--Breaker games and deep phenomena in discrete probability (in addition to their obvious combinatorial appeal) have led to an abiding interest in Maker--Breaker games. In addition to the graph-theoretic setting mentioned above, Maker--Breaker games have also been studied in arithmetic settings, where the board $X$ corresponds to some integer interval, and the winning sets are $r$-tuples of integers that are solutions to systems of linear equations in $r$ variables. We refer a reader to the 2008 monograph of Beck~\cite{Beck08} for a summary and exposition of some of the many results in the area known up to that point, and to the preprint of Kusch, Ru\'e, Spiegel and Szab\'o~\cite{KuschRueSpiegelSzabo17}  for some recent progress on hypergraph and arithmetic Maker--Breaker game, in particular establishing the tightness of the Bednarska--{\L}uczak random Maker strategies for a very general class of games.

In this paper, we investigate $(p,q)$--crossing games on rectangular grid-graphs. These differ from previous Maker-Breaker games on graphs in a number of ways: grid-graphs are far sparser than previously considered boards; the `winning sets', consisting of crossing paths, vary wildly in size, whereas in the previously studied examples they tended to all have the same size. Finally, we let both the aspect ratios ($m:n$) for our rectangular grids and the powers of both Maker and Breaker (the parameters $p$ and $q$) vary, whereas in previous games on graphs only Breaker's power varied, and a notion of aspect ratio was absent.

Our motivation for investigating crossing games comes from percolation theory. Percolation theory is a branch of probability theory concerned, broadly speaking, with the study of random subgraphs of infinite lattices, and in particular the emergence of infinite connected components. Since its inception in Oxford in the late 1950s, it has blossomed into a beautiful and rich area of research.  One of the most celebrated results in percolation theory is, without a doubt, the Harris--Kesten Theorem~\cite{Harris60, Kesten80} which we state below. 

Let $\Lambda$ denote the square integer lattice, that is, the graph on $\mathbb{Z}^2$ whose edges consist of pairs of vertices $\mathbf{v}, \mathbf{w}\in \mathbb{Z}^2$ lying at Euclidean distance $\| \mathbf{v}-\mathbf{w}\|=1$ from each other.  The $p$-random measure $\mu_p$ is, informally, the probability measure on subsets of $E(\Lambda)$ that includes each edge with probability $p$, independently of all the others. (We eschew some measure-theoretic subtleties here; for a rigorous definition of $\mu_p$ using cylinder events, see Bollob\'as and Riordan~\cite[Chapter 1]{BollobasRiordan06}.) 
\begin{theorem*}[Harris--Kesten Theorem]
	Let $\Lambda_p$ denote a $\mu_p$-random subgraph of $\Lambda$. Then
	\begin{itemize}
		\item if $p\leqslant \frac{1}{2}$, then almost surely $\Lambda_p$ does not contain an infinite component;
		\item if $p>\frac{1}{2}$, then almost surely $\Lambda_p$ contains an infinite component.
	\end{itemize}
\end{theorem*}

We began investigating \emph{Maker--Breaker percolation games}, where Maker tries to ensure the origin is contained in an infinite component, to see if some analogue of the Chvat\'al--Erd{\H o}s probabilistic intuition could hold in this setting also, despite the presence of an infinite probability space. One of the key tools in modern proofs of the Harris--Kesten theorem are the so-called \emph{Russo--Seymour--Welsh} lemmas giving bounds on the probability of crossing rectangles of various aspect ratios at $p=\frac{1}{2}$. Unsurprisingly, crossing games turned out to play an important role in our arguments when studying percolation games. In particular, the results we establish in this paper are key ingredients in the proofs of our main results on Maker--Breaker percolation games that we establish in the sequel~\cite{DayFalgasRavry19+} to the present paper.

Besides the motivation from percolation theory, we should like to stress also that crossing games are paradigmatic representatives of an important class of positional games. Indeed they are related to the older and much-studied game of Hex, and the $(1,1)$-crossing game we study here is in fact the commercially available game of Bridg-it. Which of the players wins Bridg-it under perfect plays  has been known since the late 1960s, thanks to Lehman's resolution of the more general Shannon switching game~\cite{Lehman64}. The relationship between our work in the present paper and these older games is discussed in greater detail in Sections~\ref{subsection: (1,1)-crossing} and~\ref{section: other graphs and other games}.

\section{Preliminaries}\label{section: preliminaries}

\subsection{Basic definitions and notation}\label{subsection: basic def and notation}
A graph is a pair $G=(V,E)$, where $V=V(G)$ is a set of vertices and $E=E(G)$ is a set of pairs from $V$ which form the edges of $G$.  A subgraph $H$ of $G$ is a graph with $V(H)\subseteq V(G)$ and $E(H)\subseteq E(G)$. Given $n\in \mathbb{N}$, let $[n]=\{1,2, \ldots n\}$.  In this paper, we often identify a graph with its edge-set when the underlying vertex-set is clear from context.  For the remainder of this paper, unless stated otherwise, the variables $m,n,p,q,x$ and $y$ will always be natural numbers.  

Let $\Lambda$ denote the square integer lattice, that is, the graph on $\mathbb{Z}^2$ whose edges consist of pairs of vertices $\mathbf{v}, \mathbf{w}\in \mathbb{Z}^2$ lying at Euclidean distance $\| \mathbf{v}-\mathbf{w}\|=1$ from each other.  Given $m$ and $n$, let $\Lambda_{m \times n}$ be the finite subgraph of $\Lambda$ induced by the vertex set $\{(x,y):x \in [m], y \in [n]\}$.  If $e$ is a horizontal edge in $\Lambda_{m \times n}$, that is $e = \{(x,y),(x+1,y)\}$ for some $x,y$, then we identify $e$ with its midpoint and write $e = (x+0.5,y)$.  Similarly, if $e$ is a vertical edge in $\Lambda_{m \times n}$, that is $e = \{(x,y),(x,y+1)\}$ for some $x,y$, then we denote $e$ by its midpoint and write $e = (x,y+0.5)$.  Let $S_{m \times n}$ be the graph obtained by taking $\Lambda_{m \times n}$ and removing all the edges from the set
\begin{equation}\label{equation: deleted edges}
\Big\{(1,y+0.5):  y \in [n-1]\Big\} \bigcup \Big\{(m,y+0.5):  y \in [n-1]\Big\},\notag
\end{equation}
that is, all the leftmost and rightmost vertical edges in $\Lambda_{m \times n}$.  We say a path in $S_{m \times n}$ or $\Lambda_{m \times n}$ is a \emph{left-right crossing path} if it joins some vertex $(1,y)$ on the left-hand side of the board to some vertex $(n, y')$ on the right-hand side of the board, where $y,y' \in [n]$.

We define the $(p,q)$-crossing game on $S_{m \times n}$ (respectively $\Lambda_{m\times n}$) as follows. Two players, Maker and Breaker, play in alternating turns. Maker plays first and on each of her turns claims $p$ (as-yet-unclaimed) edges of the board $S_{m \times n}$ (respectively $\Lambda_{m\times n}$); Breaker on each of his turns answers by claiming $q$ (as-yet-unclaimed) edges of the board.  The game ends if either Maker manages to claim all the edges of a left-right crossing path, in which case we declare her the winner, or if the board reaches a state where it is not longer possible for Maker to ever claim such a left-right crossing path, in which case we declare Breaker the winner.  

We shall work with crossing games on the board $S_{m \times n}$ rather than $\Lambda_{m \times n}$ for technical reasons, but for all practical purposes the two games are the same --- it can never be in a player's interest to claim an edge in $E(\Lambda_{m\times n})\setminus E(S_{m\times n})$ (so as far as winning strategies the two games are identical) and the removal of these edges makes it easier to define a dual board, as we shall shortly do below.

As a convention, we only consider the outcomes of the games under perfect play. If Maker has a winning strategy for given values $m,n,p,q$, we say that the corresponding game is a \emph{Maker win}, otherwise we say it is a \emph{Breaker win}.  Further, we follow the convention that edges claimed by Maker are coloured blue, while edges (and their dual) claimed by Breaker are coloured red.

We now define  \emph{duality} for our boards.  The dual $\Lambda^{*}$ of $\Lambda$ is the graph obtained from $\Lambda$ by shifting its vertex set by $(0.5, 0.5)$, i.e. the graph with vertex-set $\{(x +  0.5,y + 0.5): x,y \in \mathbb{Z} \} = \mathbb{Z}^{2} + (0.5,0.5)$ and edge-set consisting of all pairs of vertices lying at Euclidean distance $1$ from each other. We refer to the vertices and edges of $\Lambda^{*}$ as \emph{dual vertices} and \emph{dual edges} respectively. Just as for $\Lambda$, we identify each dual edge with its midpoint. Given an edge $e\in E(\Lambda)$, its \emph{dual} is defined to be the dual edge $e^{*}\in E(\Lambda^{*})$ such that $e$ and $e^{*}$ have the same midpoint. So for example the dual of the horizontal edge $e = (x + 0.5,y)\in E(\Lambda)$ is the vertical dual edge $e^{*} = (x + 0.5,y)^{*}$ that lies between the dual vertices $(x + 0.5,y-0.5)$ and $(x + 0.5,y+0.5)$, and the dual of the vertical edge $e = (x ,y+ 0.5)\in E(\Lambda)$ is the horizontal dual edge $e^{*} = (x ,y+ 0.5)^{*}$ that lies between the dual vertices $(x - 0.5,y+0.5)$ and $(x + 0.5,y+0.5)$.

Given a set of edges $E$ in $\Lambda$, let $E^{*} = \{e^{*} : e \in E\}$.  Given a  subgraph $\Gamma = (V,E)$ of $\Lambda$ (finite or infinite), we define its dual  $\Gamma^{*}$ to be the graph with edge-set $E^{*}$, and vertex-set consisting of all dual vertices incident to some dual edge $e^{*} \in E^{*}$.  As an example, we have $S_{m \times n}^{*}$ is a rotated and translated copy of the  graph $S_{(n+1) \times (m-1)}$. In particular, $S_{(n+1) \times n}$ is \emph{self-dual}, being isomorphic to its dual graph. Similarly, the square-integer lattice $\Lambda$ is self-dual.

In the context of Maker-Breaker crossing games, duality is important as it shows Breaker can be viewed as a ``dual Maker" aiming to build a vertical crossing dual path.
\begin{lemma}[Duality]\label{lemma: duality}
Suppose that Maker and Breaker play the $(p,q)$-crossing game on $S_{m \times n}$. At the end of the game, let $M$ be the set of edges claimed by Maker and $B$ be the set of edges claimed by Breaker.  Then precisely one of the two following statements holds:
\begin{itemize}
		\item Maker has won the game and $M$ contains a left-right crossing path  in $S_{m \times n}$,
		\item Breaker has won the game and $B^{*}$ contains a top-bottom crossing dual path of $S_{m \times n}^{*}$.
	\end{itemize}
In particular, our original game is equivalent to the game where Maker's aim is to build a left-right crossing path of $S_{m \times n}$ and Breaker's aim is to build a top-down dual crossing path of $S^{*}_{m \times n}$.	
\end{lemma}
 \begin{proof}[Proof of Lemma \ref{lemma: duality}]
While this Lemma is intuitively obvious, writing down a formal proof is not trivial. It follows however almost directly from~\cite[Lemma~1, Chapter~3]{BollobasRiordan06} which states that for any bipartition of $E(S_{m \times n})$ into disjoint sets, $E_{1}$ and $E_{2}$, precisely one of the two following statements holds:
\begin{itemize}
		\item $E_{1}$ contains a left-right crossing path of $S_{m \times n}$,
		\item $E_{2}^{*}$ contains a top-bottom crossing dual path of $S_{m \times n}^{*}$.
	\end{itemize}
Clearly if at the end of the game Maker has won, then $M$ contains a left-right crossing path in $S_{m \times n}$.  If instead at the end of the game Breaker has won, then there is no left-right crossing path in $E(S_{m \times n}) \setminus B$ and so by the above dichotomy we have that $B^{*}$ contains a top-bottom crossing path of $S_{m \times n}^{*}$.
\end{proof}
As Lemma~\ref{lemma: duality} shows, the two players in our game actually have similar aims when viewed throught the prism of duality:  Maker and Breaker are competing for resources (edges/dual edges) to build their winning sets (left-right crossing paths/top-bottom crossing dual paths). To reflect the symmetry of their competing aims, we will sometimes refer to Maker as the \textit{horizontal player}, denoting her by $\mathcal{H}$, and to Breaker as the \textit{vertical player}, and denoting him by $\mathcal{V}$. Further we will often think of Breaker as playing on the dual board and claiming dual edges on each of his turns rather than the corresponding edges of the original board (as they do in the formal game definition).

With the help of duality, one can define the boundary of a connected component in $\Lambda$ or $\Lambda^*$.
\begin{definition}[External boundary]
For a finite connected subgraph of $\Lambda$ with vertex set $C$, there is a unique infinite connected component $C_{\infty}$ of the subgraph of $\Lambda$ induced by the vertices in $\mathbb{Z}^2\setminus C$. The \emph{external boundary} $\partial^{\infty}C$ of $C$ is the collection of dual edges from $\Lambda^{*}$ that are dual to edges joining $C$ to $C_{\infty}$ in $\Lambda$. The external boundary for a set of dual vertices from a finite connected subgraph of $\Lambda^*$ is defined mutatis mutandis.
\end{definition}
It can be shown (see~\cite[Lemma~1, Chapter~1]{BollobasRiordan06}) that the external boundary $\partial^{\infty}C$ of the vertex-set $C$ of a finite connected subgraph $H$ of $\Lambda$  is a dual cycle with $C$ in its interior. A key tool in our proof of Theorem~\ref{theorem: (2q,q)-game} will be the following bound on the size of the boundary cycle in terms of the number of edges  in $H$.
\begin{lemma}[Isoperimetric Lemma]\label{lemma: isoperimetric lemma}
Let $k \in \mathbb{N}$.  If $A$ is set of $k$ edges in $\mathbb{Z}^{2}$ forming a connected component with vertex set $C$, then the dual boundary cycle $\partial^{\infty}C$ contains at most $2k+4$ dual edges.	
\end{lemma}

\begin{proof}
	We prove the lemma by induction on $k$.  
	The dual boundary cycle of a single edge has size $6$,  so our claim holds in the base case $k=1$.  Now assume that we have shown our claim holds for all components consisting of at most $k$ edges, and let $A$ be a set of $k+1$ edges forming a connected component in $\Lambda$ with vertex set $C$.

 If $A$ contains a cycle, then there exists some edge $e \in A$ such that $A \setminus \{e\}$ also gives a connected subgraph with vertex-set $C$, and so by our inductive hypothesis $\vert \partial^{\infty}C\vert\leq 2k+4$.  On the other hand if $A$ is acyclic, then the corresponding subgraph is a tree, and hence has at least one leaf (vertex of degree one). Thus there exists an edge $e \in A$ such that $A' = A \setminus \{e\}$ spans all but one vertex of $C$, say the vertex $v$.  Let $B=\partial^{\infty}(C \setminus \{v\})$; by the inductive hypothesis we know that $\vert B\vert \leqslant 2k+4$.  If $e$ is not dual to any dual edge in $B$, then $B$ is also the dual boundary cycle for  $C$, and we are done.  If on the other hand we have $e^{*} \in B$, let $f_{1},f_{2}$ and $f_{3}$ be the three dual edges that together with $e^{*}$ form the boundary cycle around the single vertex $v$. Since $e$ is the only edge of $A$ incident with $v$, none of $f_1,f_2,f_3$ lie in $A$. The union 
	\begin{equation}
	\big(B \setminus \{e^{*}\} \big) \cup \{ f_{1},f_{2},f_{3}\} \nonumber
	\end{equation}
of these dual edges with $B$ contains the external boundary $\partial^{\infty}C$ of $C$, and so this external boundary has size at most $\vert B\vert  + 2 \leqslant 2(k+1)+4 $, as required.
\end{proof}

\subsection{Elementary bounds on winning boards for the  $(p,q)$-crossing game}\label{subsection: (p,q)-crossing}
In this subsection, we make some elementary observations about winning boards for crossing games for general $(p,q)$.  We begin by giving some trivial bounds on the identity of the winner in the $(p,p)$-crossing game under optimal play on various different boards.

\begin{proposition} \label{proposition: (p,p)-game}\
\begin{enumerate}[(i)]
	\item Maker has a winning strategy for the $(p,p)$-crossing game on $S_{m \times n}$ for all $m\leqslant n +1$;
	\item Breaker has a winning strategy for the $(p,p)$-crossing game on $S_{m \times n}$ for all $m \geqslant (p+1)(n+1)$.
\end{enumerate}	
\end{proposition}
\begin{proof}
Part (i) is immediate by strategy stealing: it is enough to show that Maker can win on the self-dual board $S_{(n+1) \times n}$ (playing on a narrower board can only help Maker). Suppose for contradiction that Breaker, playing second, had a winning strategy. Then Maker can player $p$ arbitrary moves on her first turn, and from then on pretend to be Breaker playing on $S_{(n+1) \times n}^{*}$, using Breaker's putative winning strategy to respond to Breaker's actual moves (and making arbitrary moves if ever asked to claim an edge she has already claimed). Maker's initial moves can never hurt her, and thus this is a winning strategy --- contradicting our assumption that Breaker has a winning strategy, since we know this game can never end in a draw. Thus Maker must have a winning strategy.

For part (ii), it is enough to show that Breaker can win on the board $S_{(p+1)(n+1) \times n}$ (playing on a wider board can only help Breaker). We divide up this board into $p+1$ copies of $S_{(n+1) \times n}$ (plus some extra edges which we ignore). On her first move, Maker must fail to claim an edge in at least one of these copies. Thereafter Breaker plays entirely in this copy. Since $S_{(n+1) \times n}$ is self-dual and Breaker is playing first in the $(p,p)$-crossing game on this copy, he has a winning strategy. (Formally, this is not quite the $(p,p)$-crossing game: by playing on other boards, Maker could play fewer than her $p$ moves in our chosen copy of $S_{(n+1) \times n}$ in any given turn --- but this can never help her.)
\end{proof}

\begin{proposition}\label{proposition: (p, p+5r)-game}
Breaker has a winning strategy for the $(p, p+5r)$-crossing game on $S_{m\times n}$ for all $n > r$ and $m \geqslant \lceil \frac{p}{r}\rceil (n+1)$.
\end{proposition}
\begin{proof}
As before, it is enough to show that  Breaker can win on the board $S_{m\times n}$ with $m=\lceil \frac{p}{r}\rceil (n+1)$  (playing on a wider board can only help Breaker). Divide the board into $\lceil \frac{p}{r}\rceil$ copies of $S_{(n+1)\times n}$.  By our bounds on $n$ and $m$, Maker cannot have won on her first turn (since $m>p +1$). Also by the pigeonhole principle, there is one such copy on which Maker has played at most $r$ moves on her first turn.  For the remainder of the game, Breaker shall solely focus his efforts on this board, and so we may view Breaker as playing first on an $(n+1)\times n$ board where $r$ edges have been pre-emptively claimed by Maker.

Breaker shall only use his extra power of  $5r$ in his first turn, to `neutralise' Maker's edges by ensuring they can never be part of a left-right crossing path, and otherwise shall follow his winning strategy for the $(p,p)$-crossing game on an $(n+1)\times m$ board when he plays first --- a strategy which exists by Proposition~\ref{proposition: (p,p)-game} and the self-duality of $S_{(n+1)\times n}$.  (For completeness, other than on his first move, he plays arbitrary moves with his extra $5r$ edges and if ever requested to play a previously claimed edge). Provided his first-turn `neutralisation' works, Breaker will clearly win the game.

Lemma~\ref{lemma: isoperimetric lemma} established that a connected subgraph of $\mathbb{Z}^2$ with $k\geqslant 1$ edges has a dual boundary cycle of size at most $2k+4$. Further, observe that if we claim all but one of the edges in the dual boundary cycle to one of Maker's connected components $C$, then no left-right crossing path Maker makes can go through $C$, and it makes no difference to the outcome of the game if all of the edges inside $C$ had been claimed by Breaker instead. Thus to neutralise Maker's (at most) $r$ initial edge in Breaker's chosen subboard, Breaker claims all but one dual edge from the boundary cycles of each of the corresponding connected components. By our bound from Lemma~\ref{lemma: isoperimetric lemma}, this requires a total of at most $5r$ edges, which is exactly the extra power Breaker has.
\end{proof}

Clearly, the bounds on $m$ and $n$  in Proposition~\ref{proposition: (p,p)-game} and \ref{proposition: (p, p+5r)-game} are quite unsatisfactory, and we do not believe for a moment that they are tight.  See Section~\ref{section: concluding remarks} for a number of questions and conjectures pertaining to this.

\subsection{The $(1,1)$-crossing game: Bridg-it and the Shannon switching game}\label{subsection: (1,1)-crossing}

The $(1,1)$-crossing game played on $S_{m \times n}$ is also known as \textit{Bridg-it} (sometimes referred to as Bridge-it), and was first invented by David Gale.  Traditionally Bridg-it is played on a self dual grid, usually $S_{6 \times 5}$ or $S_{7 \times 6}$, however here we relax the definition to allow play on any grid-size.  Bridg-it bears some similarities to the celebrated game of \textit{Hex}, which is another positional crossing game played on the faces of a hexagonal lattice (see~\cite{HaywardRijswijck06} for a formal definition of \emph{Hex}), however Bridg-it is much simpler and better understood.

By Proposition \ref{proposition: (p,p)-game}, we know that in Bridg-it there is always a winning strategy (via strategy-stealing) for the first player, $\mathcal{H}$, when $m \leqslant n+1$.  When $m > n+1$ the vertical player $\mathcal{V}$ has a winning strategy which involves mirroring $\mathcal{H}$'s moves through an appropriate reflection of the grid.  
These two strategies (strategy-stealing and reflection strategy) have counterparts in \emph{Hex} (see e.g.~\cite{HaywardRijswijck06}). However the strategy-stealing argument does not provide an explicit winning strategy for $\mathcal{H}$, but merely proves its existence, and constructing such a strategy for $(n+1)\times n$ $\mathrm{Hex}$-boards is an extremely hard computational problem even for small $n$.

By contrast, there are several different \emph{explicit} strategies that $\mathcal{H}$ can use to win in Bridg-it whenever $m \leqslant n+1$.  The first of these to be discovered was a simple but elegant edge-pairing strategy due to Gross in 1961, see~\cite[p. 66]{Beck08} for a description.  A different strategy can be read out of a winning strategy due to Lehman~\cite{Lehman64} for a different combinatorial game, known as  the \textit{Shannon switching game}.  In addition to the crossing games studied in this paper, ideas related Lehman's winning strategy for the Shannon switching game play important role in our study of Maker-Breaker percolation games in our companion paper~\cite{DayFalgasRavry19+}. For these reasons and for completeness, we describe the Shannon switching game and its application to Bridg-it in detail below.

That strategies for the Shannon switching game may be used to construct winning strategies for $\mathcal{H}$ in Bridg-it is a well-known folklore result, which has been recorded in a number of places, see e.g.~\cite[p. 67]{Beck08}.  We present the argument below and offer the modest improvement that, on $S_{m \times n}$ with $ m \leqslant n+1$, Lehman's strategy allows $\mathcal{H}$ the freedom of picking any edge of the board on her first move and still win the entire game. (As far as we are aware, this observation has not appeared in the literature before.)

\subsubsection{The Shannon switching game}\label{subsection: Shannon switching}

The \emph{Shannon switching game} is a positional game invented by Claude Shannon.  The game is played on a triple $(G,a,b)$, where $G$ is a multigraph and $a$ and $b$ are two distinguished vertices of $G$.  At the start of the game every edge is classified as \textit{unsafe}.  Two players, \textit{Cut} and \textit{Join}, play in alternating turns in which they claim unsafe edges.  Cut plays first, and on each of his  turns picks an unsafe edge and deletes it from $G$.  Join plays second, and on each of her turns picks an unsafe edge and marks it as \textit{safe}.  The game ends when there are no unsafe edges left.  Join wins if, at the end of the game, there exists a path of safe edges from $a$ to $b$, and otherwise Cut wins if no such path exists. (Thus in our games Cut and Join correspond to Breaker and Maker respectively.)

The Shannon switching game was solved by Lehman~\cite{Lehman64}, who, for each graph, determined which of the players has a winning strategy and in addition gave an explicit description of a winning strategy in each case.
\begin{definition}
A multigraph  is \emph{$k$-positive} if it contains $k$ pairwise disjoint connected spanning subgraphs.
\end{definition}
Lehman showed that there is a winning strategy for Join in the Shannon switching game played on $(G,a,b)$ if and only if $G$ has a $2$-positive subgraph that contains both $a$ and $b$. (In fact Lehman achieved his result by generalising the Shannon switching game to a game played on matroids and solving it in that more general setting, but we will not be concerned with matroids in this paper.)

  It is the \textit{if} direction of this statement that we will need and so we reproduce its (simple) proof here. For the interested reader, a relatively short and simple proof of the \textit{only if} direction of the statement (in the language of graphs rather than matroids) was given by Mansfield in~\cite{Mansfield01}.
\begin{proposition}[Lehman~\cite{Lehman64}]\label{proposition: Join wins on 2-positive}
Suppose $a,b$ are vertices in a multigraph $G$ such that there exists a $2$-positive subgraph of $G$ containing both $a$ and $b$. Then Join has a winning strategy for the Shannon switching game played on $(G,a,b)$
\end{proposition}
\begin{proof}
Suppose $G$ has a $2$-positive subgraph that contains both $a$ and $b$.  We may pass to this subgraph and assume that $G$ is itself $2$-positive.  Let $G_{1}$ and $G_{2}$ be two edge-disjoint connected spanning subgraphs of $G$. For each $t \geqslant 0$, let $C^{t}$ be the set of the first $t$ edges that Cut deletes from $G$, and let $S^{t}$ be the set of the first $t$ edges of $G$ that Join marks as safe.  Moreover, for each $i = 1,2$ let
\begin{equation}
G_{i}^{t} = \big(G_{i} \setminus C^{t}  \big)\cup S^{t}. \nonumber
\end{equation}
Join's strategy will be to ensure that, for all $t \geqslant 0$, the graphs $G_{1}^{t}$ and $G_{2}^{t}$ are both connected spanning subgraphs of $G$.  We use induction on $t$ to show she can achieve this;  it is clear that $G_{1}^{t}$ and $G_{2}^{t}$ are both connected spanning subgraphs of $G$ when $t = 0$.  Suppose that $G_{1}^{t-1}$ and $G_{2}^{t-1}$ are both connected spanning subgraphs of $G$.  Without loss of generality, we may assume that the next edge that Cut deletes is an edge of $G_{1}$, say the edge $e=\{x,y\}$.  If $G_{1}^{t-1} \setminus \{e\}$ is still spanning and connected, then Join may play their next move arbitrarily.  If $G_{1}^{t-1} \setminus \{e\}$ is not spanning and connected, then it consists of exactly two components, one containing $x$ and the other containing $y$.  As $G_{2}^{t-1}$ is spanning and connected it contains a path from $x$ to $y$.  As $G_{1}^{t-1}$ is spanning, we must have that there exists an edge $f$ of this path that lies between the two components of $G_{1}^{t-1} \setminus \{e\}$.  As $f$ lies between the two components, $f \notin S^{t-1} \cup C^{t-1}$.  On her move, Join marks the edge $f$ as safe and adds it to $S^{t-1}$ to form $S^t$. This ensures $G_{1}^{t}$ is once again a connected spanning subgraph of $G$, as required.  Furthermore $G_{2}^{t}$ contains  $G_{2}^{t-1}$ as a subgraph, and so remains a connected spanning subgraph of $G$.  This proves our inductive statement.  When the game ends, say after Join has marked $r$ edges as safe and all other edges are unsafe, we have that $G_{1}^{r} = G_{2}^{r} = J^r$, which forms a spanning connected subgraph of $G$. In particular, there is a path of safe edges from $a$ to $b$.
\end{proof}
It is easy to extend Join's winning strategy from Proposition~\ref{proposition: Join wins on 2-positive} to the $(k,k)$-Shannon switching games where each player is allowed to claim $k$ edges on each of their turns. We leave the proof as an exercise for the reader.
\begin{proposition}\label{proposition: Join wins (k,k)-game on k-positive}
Suppose $a,b$ are vertices in a multigraph $G$ such that there exists a $(k+1)$-positive subgraph of $G$ containing both $a$ and $b$. Then Join has a winning strategy for the $(k,k)$-Shannon switching game played on $(G,a,b)$.		\qquad \qquad \qquad \qquad \qquad \qquad \qquad \qquad \qquad \qquad \qquad  \qquad \qquad \qedsymbol
\end{proposition}

\subsubsection{Winning strategy for Maker in Bridg-it}
\begin{theorem}\label{theorem: Maker winds Bridg-it with any first move}
Maker has a winning strategy for the $(1,1)$-crossing game on $S_{(n+1)\times n}$ (i.e. the game of Bridg-it) that allows her to choose any edge she wants on her first move.
\end{theorem}
\begin{proof}
We begin by $2$-colouring the edges of $S_{(n+1) \times n}$. All horizontal edges (i.e. all edges of the form $(x, y+0.5)$) are assigned the colour green, while all vertical edges (i.e. all edges of the form $(x, y+0.5)$) are coloured orange.  The horizontal player $\mathcal{H}$ (Maker)  then picks an arbitrary edge $e$  as her first edge and colours it blue. Based on the choice of $e$, we define a set $A$ of  green edges which $\mathcal{H}$ shall recolour and use in her strategy.

If $e$ is a green edge, we let $A$ be any set of $n-1$ green edges such that no two edges in $A\cup\{e\}$ have the same $x$-coordinate, and no two edges in $A\cup\{e\}$ have the same $y$-coordinate.  If instead $e$ was an orange edge, say $e = (x,y+0.5)$, then let $f_{1} = (x+0.5,y)$ and $f_{2} = (x-0.5,y+1)$.  Let $A'$ be any set of $n-2$ green edges such that no two edges in $A' \cup \{f_{1},f_{2}\}$ have the same $x$ coordinate, and no two edges in $A' \cup \{f_{1},f_{2}\}$ have the same $y$ coordinate.  Let $A = A' \cup \{f_{1},f_{2}\}$.

In either case, we recolour all the edges in $A$ with the colour orange.  Let $G$ be the graph formed from $S_{(n+1) \times n}$ by contracting all vertices $(1,y)$ into a single vertex $a$, and contracting all vertices of the form $(n+1,y)$ into a single vertex $b$.   There is a one-to-one correspondence between the edges of $S_{(n+1) \times n}$ and $G$, so we may consider the colouring that $G$ inherits from $S_{(n+1) \times n}$.  Let $e'$ be in edge in $G$ that corresponds to the edge $e$ in $S_{(n+1) \times n}$, i.e. the unique blue edge in the graph. 
Let $G_{1}$ be the subgraph of $G$ whose edge set consists of the set of green edges together with the unique blue edge $e'$.  Similarly, let $G_{2}$ be the subgraph of orange edges together with the blue edge $e'$. 
See Figure \ref{Fig2.1} for an example of these graphs when the first edge that $\mathcal{H}$ chose was an orange edge.

\begin{figure}[ht]
    \centering
	\includegraphics[scale=1]{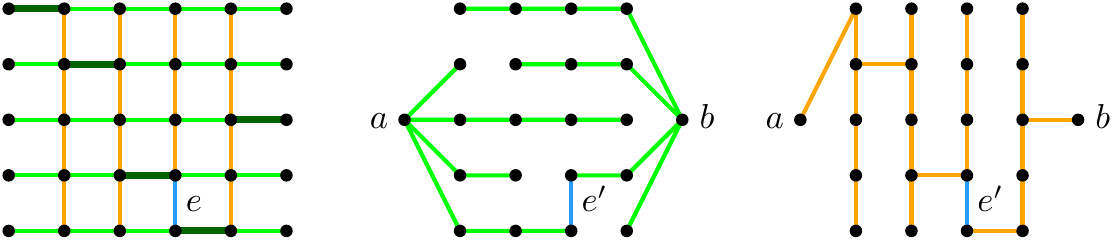}
	\caption{An example of the initial colouring $\mathcal{H}$ uses for her winning strategy.  In the first graph, the blue edge is first edge that $\mathcal{H}$ plays and the set of dark green edges are an example of a suitable set $A$ for the recolouring.  The second and third diagrams show the graphs $G_{1}$ and $G_{2}$ respectively, which are the graphs that arise from the colouring inherited by $G$.}
	\label{Fig2.1}
\end{figure}

It is easy to see that $G_{1}$ and $G_{2}$ are both connected spanning subgraphs of $G$ containing $a$ and $b$.  Moreover, the only edge these two graphs share in common is the edge $e'$ that Maker has claimed on her first turn.  Thus, if we consider this edge as `safe', then we know by Propostion~\ref{proposition: Join wins on 2-positive} that Join has an explicit winning strategy on this graph when playing the Shannon switching game, where the two distinguished vertices $a$ and $b$ are the left- and right-most vertices.  Thus $\mathcal{H}$'s strategy in Bridg-it is simply to lift Join's strategy from the Shannon switching game on $G$ to the $(1,1)$-crossing game on $S_{(n+1) \times n}$.  At the end of the Shannon switching game on $G$ we know that Join has constructed a path of safe edges from $a$ to $b$.  When lifted back to $S_{(n+1) \times n}$ this path is a left-right crossing path of $S_{(n+1) \times n}$, as required.\end{proof}
	

\section{The $(2q-1,q)$-crossing game: Breaker wins on sufficiently long boards}\label{section: (2q-1,q)-game}
In this section we prove Theorem \ref{theorem: (2q-1,q)-game}, which states that if $m$ is sufficiently large (with respect to $q$ and $n$), then Breaker, also referred to as the vertical player $\mathcal{V}$, has a winning strategy for the $(2q-1,q)$-crossing game on $S_{m \times n}$.
\begin{proof}[Proof of Theorem \ref{theorem: (2q-1,q)-game}]
Let $T$ be the number of edges in $S_{(n+1) \times n}$, that is $T = n^{2} + (n-1)^2$.  Let \begin{equation}
m_{0} = m_{0}(q,n) = (n+1)\big((6q-2)^{T} + 2q-1\big). \nonumber
\end{equation}
We split the board $S_{m_{0} \times n}$ into $(6q-2)^{T} + 2q-1$ disjoint copies of $S_{(n+1) \times n}$, which we call strips.  Recall our convention that edges claimed by Breaker are coloured red. At any point during the game, we say a strip is $k$\textit{-valid} if it contains exactly $k$ red edges and is in a winning position for $\mathcal{V}$ in the $(1,1)$-crossing game on $S_{(n+1) \times n}$ when $\mathcal{V}$ plays second.  We say a strip is $k$\textit{-neutral} if it contains exactly $k$ red edges and is in a winning position for $\mathcal{V}$ in the $(1,1)$-crossing game when $\mathcal{V}$ gets to play first.  If a strip is neither $k$-valid nor $k$-neutral for any integer $k$ we say that it is \textit{invalid}.  Note that if a strip is $k$-valid, then it is also $k$-neutral.

We know by Proposition \ref{proposition: (p,p)-game} that every strip is $0$-neutral at the start of the game.  The game begins with $\mathcal{H}$ playing edges in up to $2q-1$ different strips, possibly making them invalid in the process.  At this point, the vertical player $\mathcal{V}$'s strategy will proceed in $T+1$ phases, with phase $0$ starting after $\mathcal{H}$'s initial turn.  For each $k \in\{ 0,1,\ldots,T\}$, $\mathcal{V}$'s strategy will ensure that at the beginning of phase $k$ (i) it is $\mathcal{V}$'s turn to play, and (ii) there are at least $(6p-2)^{T - k}$ $k$-neutral strips. Note that this implies that at the start of phase $T$ there will be at least one $T$-neutral strip, which by definition must contain a path of red dual edges from the top of the strip to the bottom of the strip, and thus $\mathcal{V}$ wins the game.


Clearly (i) and (ii) both hold for $k = 0$.  Let us now show that if (i) and (ii) both hold at the beginning of phase $k$, then $\mathcal{V}$ can ensure they both hold at the beginning of phase  $k+1$ too.  On each turn in phase $k$, the vertical player $\mathcal{V}$ will choose $q$ different $k$-neutral strips and play a single edge in each that turns these $k$-neutral strips into $(k+1)$-valid strips.  The horizontal player $\mathcal{H}$ can now distribute their $2q -1$ edges among all of the strips. Each edge that $\mathcal{H}$ plays can either turn a $(k+1)$-valid strip into a $(k+1)$-neutral strip, or turn a $k$-neutral or $(k+1)$-neutral strip into an invalid one (or can be played in another kind of strip, in which case we ignore it).

For each $t \in \mathbb{Z}_{\geqslant 0}$, let $A_{t}=A_t(k)$ be the number of $(k+1)$-valid strips after a combined total of $t$ edges have been claimed by the two players in phase $k$ of the game (where for convenience we imagine the two players play the edges on their turn in some arbitrary order).  Similarly, let $B_{t}=B_t(k)$ be the number of $(k+1)$-neutral strips after a combined total of $t$ edges have been played by the two players in phase $k$.  Let $R_{t}=R_t(k)$ be defined by $R_t= 2A_{t} + B_{t}$.  

How does $R_t$ vary with $t$? If the next edge to be claimed is one of $\mathcal{V}$'s, then $R_{t+1} = R_{t} + 2$.  On the other hand, if the next edge to be claimed is one of $\mathcal{H}$'s, then $R_{t+1} \geqslant R_{t} -1$.  As the two players claim a combined total of $3q-1$ edges on each turn of the game, we have that $R_{r(3q-1)} \geqslant r$ for all $r \in \mathbb{Z}_{\geq0}$, until either phase $k$ ends or $\mathcal{V}$ runs out of $k$-neutral strips.

Now $\mathcal{V}$ decides that phase $k$ ends (and phase $k+1$ begins) when $R_{r(3q-1)} \geqslant 2(6q-2)^{T-k-1}$ for some $r \in \mathbb{Z}_{\geqslant 0}$.  Note that after $\mathcal{H}$ and $\mathcal{V}$ have both completed their turns, the number of $k$-neutral strips has decreased by at most $3q-1$.  Thus, as the number of $k$-neutral strips at the start of phase $k$ is at least $(6q-2)^{T-k}$, we know that the number of $k$-neutral strips for $\mathcal{V}$ to play in will not run out before $R_{r(3q-1)} \geqslant 2(6q-2)^{T-k-1}$.  As $R_{r(3q-1)} \geqslant 2(6q-2)^{T-k-1}$ we have that the number of $(k+1)$-neutral strips at the start of phase $k+1$ is at least $(6q-2)^{T-k-1}$ and further that it is $\mathcal{V}$'s turn to play, so that (i) and (ii) both hold as required.
\end{proof}

\section{The $(2q, q)$-crossing game: Maker wins on  arbitrarily long and narrowest possible boards}\label{section: (2q,q)-game}

In this section we prove Theorem \ref{theorem: (2q,q)-game}, which states that if $n \geqslant q+1$, then Maker, also referred to as the horizontal player $\mathcal{H}$, has a winning strategy for $(2q,q)$-crossing-game on $S_{m \times n}$, for any $m\in \mathbb{Z}_{\geqslant 0}$.  Note that the condition $n \geqslant q+1$ is necessary as if $n \leqslant q$, then $\mathcal{V}$ could win the game in a single turn.  We in fact prove a more general result, showing $\mathcal{H}$ can win the \emph{$q$-double-response game} (defined below) --- this will not complicate the argument, and the greater generality will  allow us to apply these results to the study of percolation games in the sequel to this paper~\cite{DayFalgasRavry19+}. A key idea in our proof will be to consider a third game, the \emph{secure game}, where $\mathcal{V}$ plays one edge at at time but is given the extra power of reclaiming some of $\mathcal{H}$'s edges.  This will allow us to treat a $(2q,q)$ game like a $(2,1)$ game, which is much more amenable to analysis, and we shall show that even with $\mathcal{V}$'s extra powers, $\mathcal{H}$ still has a winning strategy.
  
Let $S_{\infty \times n}$ be the infinite subgraph of $\Lambda$ induced by the vertex set $\{(x,y):x \in \mathbb{Z}, y \in [n]\}$.  The $q$\textit{-double-response game} is a game played by two players, a horizontal player $\mathcal{H}$ and a vertical player $\mathcal{V}$, on the edges of $S_{\infty \times n}$.  The game being with $\mathcal{V}$ playing first. On each turn $t$, $\mathcal{V}$ picks an integer $r_{t} \in [q]$ and then claims $r_{t}$ as-yet unclaimed edges in $S_{\infty \times n}$ for himself; then $\mathcal{H}$ answers by claiming $2r_t$  as-yet unclaimed edges in response to $\mathcal{V}$'s move. In this game, $\mathcal{V}$'s aim is to claim a set of edges corresponding to a top-bottom crossing path of dual edges, and we say $\mathcal{V}$ wins if he is able to do so.  The horizontal player $\mathcal{H}$'s aim is to prevent this from ever happening, and we say $\mathcal{H}$ wins the game if she is able to do so.  We remark that throughout this section we will always view the game through the lens of duality, so that $\mathcal{V}$ always claims dual edges.

We will show that if $n \geqslant q+1$, then $\mathcal{H}$ has a winning strategy for the $q$-double-response game.  Clearly this implies $\mathcal{H}$  has a winning strategy in $(2q,q)$-crossing-game on $S_{m \times n}$, playing as Maker (and even surrendering her first move).  Thus Theorem~\ref{theorem: (2q,q)-game} is immediate from the following.
\begin{theorem}\label{theorem: double-response}
If $n \geqslant q+1$, then $\mathcal{H}$ can win the $q$-double-response game on $S_{\infty \times n}$.
\end{theorem}

Before we prove Theorem \ref{theorem: double-response}, let us sketch the main ideas behind the proof and give some preliminary definitions.  We define an \textit{arch} to be a path of edges that starts and ends at a bottommost vertex or starts and ends at a topmost vertex in $S_{\infty \times n}$.  Similarly, we define a \textit{dual arch} to be a path of dual edges that starts and ends at a bottommost dual vertex or starts and ends at a topmost dual vertex in $S_{\infty \times n}^{*}$.

We may assume that $\mathcal{V}$ never claims either a dual edge as red that would create a cycle of red dual edges or a dual arch of red dual edges. Indeed, if $\mathcal{V}$ plays such a dual edge $e$, then at any stage later on in the game, if there exists a path $P$ of red dual edges from the top of the grid to the bottom of the grid, then there still exists such a path if we remove $e$. In particular, the result of the game cannot depend on the identity of the player who claimed $e$ (or equivalently its dual $e^*$). Therefore if such a dual edge $e$ was played, we can ignore it and pretend $\mathcal{V}$ has claimed some other edge.

A key ingredient in the proof will be Lemma~\ref{lemma: isoperimetric lemma}, which stated that if $A$ is a set of $k$ edges in $\mathbb{Z}^{2}$ that form a connected component $C$, then the dual boundary cycle to $C$ contains at most $2k+4$ dual edges.	 While we do not use Lemma \ref{lemma: isoperimetric lemma} directly, it is the ``explanation'' for why our proof works, and it will be helpful for the reader to keep it in mind throughout this section.

Suppose that $\mathcal{H}$ was able to ensure that, at the end of each of her turns she has claimed every edge of every boundary cycle of every component created by $\mathcal{V}$'s red dual edges.  If so, then as any top-bottom dual crossing path needs at least $n\geqslant q+1$ dual edges, $\mathcal{V}$ would be unable to win on any turn --- and so $\mathcal{H}$ clearly wins the game.  Unfortunately $\mathcal{H}$ cannot always claim all the edges of every boundary cycle.  For example, if $\mathcal{V}$ plays $q$ pairwise disjoint and sufficiently spaced-out dual edges, then $\mathcal{H}$ would need $6q$ edges to claim all the edges in each of the boundary cycles of the $q$ components formed by these dual edges.  However what $\mathcal{H}$ can hope for, given Lemma~\ref{lemma: isoperimetric lemma},  is to claim all but at most $4$  edges in every boundary cycle of every component of red dual edges.  Our strategy will show that $\mathcal{H}$ can indeed do this, and can do it in such a way that $\mathcal{V}$ will never be able to create a top-bottom dual crossing path, even by connecting up components created over many different turns.  To make this precise, we need some definitions.
\begin{definition}[Components]
For the rest of this section, whenever we refer to a \textit{component} we mean, at a given stage in the game, a maximal set of two or more dual vertices connected by some path of red dual edges.  We say a component is a \textit{top component} if it contains at least one dual vertex from the top of the grid, while we say a component is a \textit{bottom component} if it contains at least one dual vertex from the bottom of the grid.  If a component is neither a top nor a bottom component, then we say that it is a \textit{floating component}.	
\end{definition}
\begin{definition}[Interiors]
Given a closed cycle $C$ of blue edges, we define the \textit{interior} of $C$ to be the set of dual vertices $v$ such that every dual path from $v$ to a top- or bottommost dual vertex must contain the dual of an edge in $C$.  If $A$ is an arch that starts and ends at a bottommost vertex, then we define the \textit{interior} of $A$ to be the set of dual vertices $v$ such that every dual path from $v$ to a topmost vertex contains the dual of an edge in $A$. Similarly, if $A$ is an arch that starts and ends at a topmost vertex, then we define the \textit{interior} of $A$ to be the set of dual vertices $v$ such that every dual path from $v$ to a bottommost vertex contains the dual of some edge in $A$.
\end{definition}
We now come to the key definition of \emph{brackets}. Underpinning our strategy for $\mathcal{H}$ is the fact that she can ensure the $4$ edges in a component's boundary cycle she is unable to claim have a nice form, namely that of one of the following brackets.  See Figure \ref{Fig1} for a picture of these different bracket types, together with their corners and interior dual vertices, as defined below.
\begin{definition}[Brackets]
We say the edges 
\begin{equation}
\{(x+0.5,y),(x+1.5,y),(x+2,y+0.5),(x+2,y+1.5)\} \nonumber
\end{equation}
form a \emph{bracket} of \emph{Type $1$} if none of them are red.  We call the vertices $(x,y)$ and $(x+2,y+2)$ the \emph{corners} of the bracket, and we call the dual vertices $(x+0.5,y+0.5),(x+1.5,y+0.5)$ and $(x+1.5,y+1.5)$ the \emph{interior} dual vertices of the bracket.  We say the edges 
\begin{equation}
\{(x+0.5,y),(x+1,y+0.5),(x+1.5,y+1),(x+2,y+1.5)\} \nonumber
\end{equation}
form a \emph{bracket} of \emph{Type $2$} if none of them are red.  We call the vertices $(x,y)$ and $(x+2,y+2)$ the \textit{corners} of the bracket, and we call the dual vertices $(x+0.5,y+0.5)$ and $(x+1.5,y+1.5)$ the \textit{interior} dual vertices of the bracket.  
We say the edges 
\begin{equation}
\{(x,y - 0.5),(x+0.5,y-1),(x+1,y-0.5),(x+1,y+0.5)\} \nonumber
\end{equation}
form a \emph{bracket} of \emph{Type $3^{+}$} if none of them are red.  We call the vertices $(x,y)$ and $(x+1,y+1)$ the \textit{corners} of the bracket, and we call the dual vertices $(x+0.5,y-0.5)$ and $(x+0.5,y+0.5)$ the \textit{interior} dual vertices of the bracket.  Finally, we say the edges 
\begin{equation}
\{(x+0.5,y),(x+1.5,y),(x+2,y+0.5),(x+1.5,y+1)\} \nonumber
\end{equation}
form a \emph{bracket} of \emph{Type $3^{-}$} if none of them are red.  We call the vertices $(x,y)$ and $(x+1,y+1)$ the \textit{corners} of the bracket, and we call the dual vertices $(x+0.5,y+0.5)$ and $(x+1.5,y+0.5)$ the \textit{interior} dual vertices of the bracket.  
\end{definition}
\begin{figure}[ht]
	\centering
	\includegraphics[scale=1]{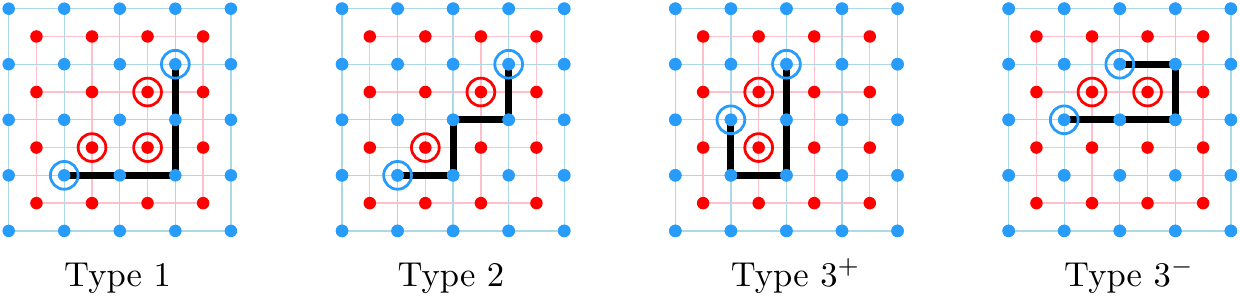}
	\caption{The four different bracket types, represented by the black edges in each picture.  For each bracket the circled blue vertices are its corners, while the circled red dual vertices are its interior dual vertices.}
	\label{Fig1}
\end{figure}	
\begin{remark}\label{remark: bracket symmetry}
	Any bracket of Type $1$ or Type $2$ is preserved under the reflection that switches its two corners.  Moreover, if you reflect a bracket of Type $3^{+}$ through the reflection that switches its two corner vertices, then you end up with a bracket of Type $3^{-}$, and vice-versa.  More generally the set of brackets is closed under reflections through lines parallel to $x+y=0$.
\end{remark}
We will make use of the above remark to reduce the amount of (tedious but necessary) case-checking required in the proof of Theorem~\ref{theorem: double-response}. 
\begin{definition}[Secure floating component]	
At any stage of the game, we say a floating component $C$ is \emph{secure} if 
\begin{enumerate}[(i)]
	\item there exists a bracket $B$ (of any type) and a path $P$ of blue edges such that the corners of $B$ are the end points of the path $P$;
	\item the interior dual vertices of the bracket $B$ are in $C$;
	\item $C$ is in the interior of the cycle formed by $P\cup B$;
	\item for every edge $f \in P$, at least one of the dual vertices of the dual edge $f^{*}$ is in $C$. 
\end{enumerate}
\end{definition}
\begin{definition}[Secure bottom component, gate]
If $C$ is a bottom component, then  (as  we may assume $\mathcal{V}$ never plays a dual edge that creates a dual arch of red dual edges)  $C$ contains a unique bottommost dual vertex  $v = (x+0.5,0.5)$.  We say that $C$ is \emph{secure} if there exist a non-red edge $e = (x',1.5)$ for some $x' \in \mathbb{Z}$ with $x' \geqslant x+1$, and a path $P$ of blue edges from the vertex $(x',2)$ to the vertex $(x,1)$, such that 
\begin{enumerate}[(i)]
	\item if $x' > x+1$, then $e$ is in fact a blue edge;
	\item $C$ is contained in the interior of  $P \cup \{e\}$;
	\item for every edge $f \in P$, at least one of the dual vertices of the dual edge $f^{*}$ is in $C$.
\end{enumerate}
We say that the edge $e$ is the bottom component $C$'s \emph{gate}, and if this edge $e$ is blue, then we say that $C$ is \emph{extra secure}.
\end{definition}
\begin{definition}[Secure top component, gate]
If $C$ is a top component, then (as we may assume that $\mathcal{V}$ never plays a dual edge that creates a dual arch of red dual edges)  $C$ contains a unique topmost dual vertex  $v = (x+0.5,n+0.5)$.  We say that $C$ is \emph{secure} if there exist a non-red edge $e = (x',n-0.5)$ for some $x' \in \mathbb{Z}$ with $x' \geqslant x+1$, and a path $P$ of blue edges from the vertex $(x',n-1)$ to the vertex $(x,n)$, such that 
\begin{enumerate}[(i)]
	\item if $x' > x+1$ then $e$ is in fact a blue edge;
	\item $C$ is contained in the interior of  $P \cup \{e\}$
	\item for every edge $f \in P$, at least one of the dual vertices of the dual edge $f^{*}$ is in $C$.
\end{enumerate}
We say that the edge $e$ is the top component $C$'s \emph{gate}, and if this edge $e$ is blue, then we say that $C$ is \emph{extra secure}.
\end{definition}

We say the grid is \emph{secure} at a given stage of the game if every component is secure.  Note that if the grid is secure, then no component can simultaneously be a top component and a bottom component, i.e. there is no top-bottom red dual crossing path.  See Figure \ref{Fig2} for an example of a grid in a secure position.  
\begin{figure}[ht]
    \centering
	\includegraphics[scale=1.3]{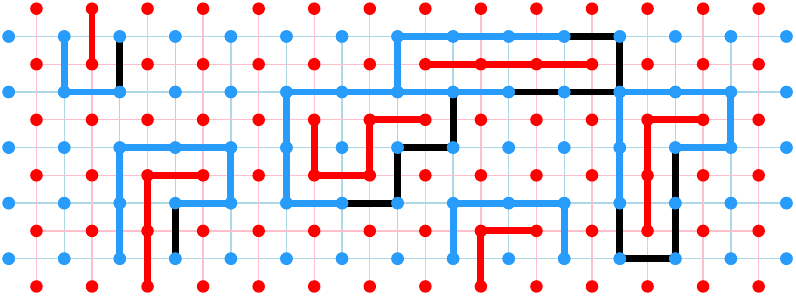}
	\caption{An example of a section of the grid $S_{\infty \times 5}$ in a secure position.  The black edges are the unclaimed edges that form the brackets or gates of the various components.}
	\label{Fig2}
\end{figure}

\begin{lemma}\label{Lemma1}
Let $n \geqslant q+1$.  If the grid is in a secure position at the start of $\mathcal{V}$'s turn in the $q$-double-response game played on $S_{\infty \times n}$, then $\mathcal{V}$ cannot win in a single turn. 
\end{lemma}
\begin{proof}
Let us suppose that the grid is in a secure position and that $\mathcal{V}$ claims $l$ dual edges and thereby creates a path of red dual edges $P$ that connects the top of the grid to the bottom of the grid.  Let $\{e_{1},\ldots,e_{l}\}$ be the dual edges in $P$ that $\mathcal{V}$ claimed in the order that they appear when one travels along $P$ from the bottom of the grid to the top of the grid.  We may assume that none of the dual edges in $\{e_{1},\ldots,e_{l}\}$ has both of its end-points in the same component (before $\mathcal{V}$ takes his turn), as such an edge would be superfluous with respect to creating a top-down dual crossing path.  We will show that $l \geqslant n$, which in turn proves the lemma as $n \geqslant q+1$.

For each dual edge $e_{i}$, let $v_{i}^{-}$ and $v_{i}^{+}$ be the end dual vertices of $e_{i}$ such that, when travelling along $P$ from the bottom of the grid to the top of the grid, one traverses $v_{i}^{-}$ before $v_{i}^{+}$.  For each such dual edge, let $x_{i} \in \mathbb{Z}$ and $y_{i}\in \mathbb{N}$ be such that $v_{i}^{+} = (x_{i}+0.5,y_{i} + 0.5)$.  We will show by induction on $i$ that $y_{i} \leqslant i$.  The statement is clear when $i = 1$ as either $e_{1}$ must be a dual edge that meets the bottom of the grid, and so $y_{1} = 1$, or $e_{1}$ meets a bottom component, say $C$.  As $C$ is secure, we must have that $e_1$ lies across $C$'s gate, say the edge $f = (x,1.5)$.  As such we have that $e_1 = f^{*}$ and so $y_{1} = 1$.

Now suppose that $y_{k} \leqslant k$ and consider the dual edge $e_{k+1}$.  If $e_{k}$ and $e_{k+1}$ have a dual vertex in common, then it is clear that $y_{k+1} \leqslant y_{k} + 1$ and so we are done.  If $e_{k}$ and $e_{k+1}$ do not share a dual vertex, then there must be some floating component $C$ such that both $e_{k}$ and $e_{k+1}$ are adjacent to a vertex in $C$, and the dual edges $e_{k}$ and $e_{k+1}$ each lie across some edge from  $C$'s bracket $B$.  If $B$ is a bracket of Type $1$ or Type $2$ with corners $(x,y)$ and $(x+2,y+2)$, then we must have that $y_{k} \geqslant y$ and $y_{k+1} \leqslant y + 1$.  Similarly, if $B$ is a bracket of Type $3^{+}$ with corner vertices $(x,y)$ and $(x+1,y+1)$, then we must have that $y_{k} \geqslant y-1$ and $y_{k+1} \leqslant y$.  Finally,  if $B$ is a bracket of Type $3^{-}$ with corner vertices $(x,y)$ and $(x+1,y+1)$, then we must have that $y_{k} \geqslant y$ and $y_{k+1} \leqslant y+1$. In all cases we have shown that $y_{k+1} \leqslant y_{k} + 1$ and so we have proved our inductive claim.

We now show that we have $l \geqslant n$.  Suppose for contradiction that $l \leqslant n-1$, and consider the dual edge $e_{l}$.  We must have that $e_{l}$ meets either the top of the grid or  a top component.  As we showed above that $y_{l} \leqslant l  \leqslant n-1 $, we can rule out the first of these two possibilities:  $e_{l}$ cannot meet the top of the grid. Thus $e_{l}$ meets a top component.  Moreover, as this top component is secure, $e_{l}$ is a horizontal dual edge, $y_{l} = n-1$, and $v_{l}^{-}$ lies to the right of $v_{l}^{+}$.  Note, we cannot have $v_{l}^{-} = v_{l-1}^{+}$ as this would contradict the fact that $y_{l-1} \leqslant l-1 \leqslant n-2$.  Thus the vertex $v_{l}^{-}$ must be part of some floating component $C$, and the dual edge $e_{l}$ must lie across $C$'s bracket $B$.  However, the only way this would be possible is if $B$ were a bracket of Type $3^{+}$ with corner vertices $(x_{l}+1, n)$ and $(x_{l}+2,n+1)$.  If this were the case, we must have that the dual vertex $(x_{l} + 1.5, n + 0.5)$ is also part of $C$, as it is an interior dual vertex of the bracket $B$.  However this would tell us that $C$ is a top component and not a floating component.  Therefore no such component $C$ can exist, which gives the desired contradiction. 
\end{proof}

Lemma~\ref{Lemma1} tells us that if the grid is secure at the start of $\mathcal{V}$'s turn, then it is not possible for $\mathcal{V}$ to win in a single turn.  We now show that if the grid is secure at the start of $\mathcal{V}$'s turn, then, after $\mathcal{V}$ has claimed $r \leqslant q$ edges, $\mathcal{H}$ can return the grid to a secure state by placing at most $2r$ blue edges.  This immediately implies that $\mathcal{H}$ wins the $q$-double-response game on $S_{\infty \times n}$, whenever $n \geqslant q+1$, and thus proves Theorem~\ref{theorem: double-response}.

 To show that $\mathcal{H}$ can always return the grid to a secure state at the end of each of her turns, we introduce the \textit{secure game}.  The main idea behind this game is that it allows $\mathcal{H}$ to respond to $\mathcal{V}$'s edges one at a time.

\begin{definition}[Secure game]
The secure game is played by $\mathcal{H}$ and $\mathcal{V}$ on the graph $S_{\infty \times n}$. At any point in the game, some edges will be unclaimed, some red (claimed by $\mathcal{V}$), some blue (claimed by $\mathcal{H}$), and some will have become blue double-edges (claimed twice by $\mathcal{H}$).

On each of his turns, $\mathcal{V}$ claims an edge  and colours it red. The edge he claims may be unclaimed, or may already be a blue edge or a blue double-edge, in which case $\mathcal{V}$ breaks these blue edges and replaces them by a red single edge.  However $\mathcal{V}$'s choice of an edge (regardless of whether it is an unclaimed edge, a blue edge or a blue double-edge) is subject to three restrictions:
\begin{enumerate}[(a)]
	\item $\mathcal{V}$ is not allowed to claim an edge  if doing so would create a red dual cycle or a red dual arch;
	\item $\mathcal{V}$ is not allowed to claim an edge if doing so connects a top component to a bottom component;
	\item if $C$ is a floating component and $P$ is the path of blue edges that helps secure $C$, then $\mathcal{V}$ is not allowed to claim an edge from $P$ if doing so turns $C$ into either a top or a bottom component.
\end{enumerate}

Once $\mathcal{V}$ has played an edge $e$, $\mathcal{H}$ responds by claiming $b+2$ edges and colouring them blue, where $b$ is the number of blue edges broken by $e$, counting multiplicity. Thus $\mathcal{H}$ may respond with $2$, $3$ or $4$ edges.

At any stage of this game, we say the grid is \emph{secure} if two conditions are met.  The first condition is that the grid is in a secure position as far as the $q$-double-response game is concerned (treating all blue double-edges as blue simple edges for that purpose).  The second condition is that if  $C$ and $C'$ are distinct red components and $P$ and $P'$ are the blue paths securing them, then every edge in the intersection $P\cap P'$ is a blue double-edge.

We say $\mathcal{H}$ wins the game if she can ensure that at the end of each of her turns the game is in a secure position (i.e. the board remains secure however long we play). Otherwise, we say $\mathcal{V}$ wins.
\end{definition}

\begin{lemma}\label{Lemma2}
For any $n\geqslant 2$, the horizontal player $\mathcal{H}$ has a winning strategy for the secure game on the grid $S_{\infty \times n}$.
\end{lemma}
\begin{proof}
We will show how $\mathcal{H}$ can win the secure game by supposing the grid is in a secure position, and describing how $\mathcal{H}$ should respond to any dual edge that the vertical player $\mathcal{V}$ claims.  Suppose that $\mathcal{V}$ has played his single red dual edge $e=\{v_{1},v_{2}\}$.  We split into a number of different cases, determined by whether or not the dual vertices $v_1$ and $v_2$ are part of pre-existing components.  Some of these cases are then split into further sub-cases depending on whether or not $e$ lies across an existing blue edge or a bracket.

\begin{case}\label{Case1}
Before $e$ is played, neither of the dual vertices $v_{1}$ or $v_{2}$ are part of a component.
\end{case}
If one of $v_{1}$ or $v_{2}$ is a bottommost dual vertex, then, without loss of generality, we can write $v_{1} = (x+0.5,0.5)$ and $v_{2} = (x+0.5,1.5)$ for some $x \in \mathbb{Z}$.  In this case we know that the three edges $(x,1.5)$,  $(x+0.5,2)$ and $(x+1,1.5)$ are not red (as otherwise $v_{2}$ would be part of some component before $e$ was played).  The horizontal player $\mathcal{H}$ plays the edges $(x,1.5)$ and $(x+0.5,2)$.  The grid is still secure as the new component that $\mathcal{V}$ created is a bottom component and is secured by the blue path $P = \{(x,1.5),(x+0.5,2)\}$ and the gate $G = \{(x+1,1.5)\}$.

Similarly, if one of the $v_{i}$ is a topmost dual vertex, then, without loss of generality, we can write $v_{1} = (x+0.5,n + 0.5)$ and $v_{2} = (x+0.5,n - 0.5)$ for some $x\in \mathbb{Z}$.  The edges $(x,n- 0.5), (x+0.5,n-1)$ and $(x+1,n - 0.5)$ are not red (as otherwise $v_{2}$ would be part of some component before $e$ was played).  The horizontal player $\mathcal{H}$ plays the edges $(x,n-0.5)$ and $(x+0.5,n-1)$.  The grid is still secure as the new component that $\mathcal{V}$ created is a top component and is secured by the blue path $P = \{(x,n-0.5),(x+0.5,n-1)\}$ and the gate $G = \{(x+1,n-0.5)\}$.

If neither of the dual vertices $v_{1}$ or $v_{2}$ is a bottom or topmost dual vertex, then the dual edge $e=\{v_1,v_2\}$ forms a new floating component.  If $e$ is a vertical dual edge, say $e = (x+0.5,y)^{*}$, then none of the six following edges are red (as otherwise one of $v_{1}$ or $v_{2}$ would have been part of a component before $e$ was played):
$(x,y+0.5)$, $(x+0.5,y+1)$, $(x+1,y+0.5)$, $(x+1,y-0.5)$, $(x+0.5,y-1)$, $(x,y-0.5)$.
The horizontal player $\mathcal{H}$ now claims the edges $(x,y+0.5)$ and $(x+0.5,y+1)$ and colours them blue.  The grid is now secure as the new component created by $\mathcal{V}$ is a floating component secured by the blue path $P = \{(x,y+0.5),(x+0.5,y+1)\}$ and the bracket $B$ of Type $3^{+}$ with corner vertices $(x,y)$ and $(x+1,y+1)$.

Similarly, if $e$ is a horizontal dual edge, say $e = (x,y+0.5)^{*}$, then none of the following edges are red (as otherwise one of $v_{1}$ or $v_{2}$ would have been part of a component before $e$ was played): $(x-1,y+0.5)$, $(x-0.5,y+1)$, $(x+0.5,y+1)$, $(x+1,y+0.5)$, $(x+0.5,y)$, $(x-0.5,y)$. The horizontal player now claims the edges $(x-1,y+0.5)$ and $(x-0.5,y+1)$ and colours them blue.  The grid is now secure as the new component created by $\mathcal{V}$ is a floating component secured by the blue path $P = \{(x-1,y+0.5),(x-0.5,y+1)\}$ and the bracket $B$ of Type $3^{-}$ with corner vertices $(x-1,y)$ and $(x,y+1)$.

We have now dealt with all possibilities in Case \ref{Case1}.

\begin{case}\label{Case2}
Before $e$ is played, the vertex $v_{1}$ is part of some component $C$ while the vertex $v_{2}$ is not part of any component.
\end{case}

Let $P$ be the path of blue edges that secures the component $C$.  If $C$ is a floating component, let $B$ be the bracket that, together with $P$, secures $C$.  If instead $C$ is a bottom or top component, let $G$ be the be the gate that helps secure $C$.  If $v_{2}$ lies in the interior of the cycle formed by $P \cup B$, or the arch formed by $P \cup G$, then the grid is still secure after $e$ has been played, and so $\mathcal{H}$ may play her edges arbitrarily.  If $v_{2}$ is not in the interior of the cycle formed by $P \cup B$ or the arch formed by $P \cup G$, then we note that the edge $e$ must lie across either $P$, $B$ or $G$, as if this were not the case, then $v_{2}$ would be a vertex that contradicts the fact that $C$ was secure before $\mathcal{V}$'s move.

We first deal with the case that $e$ crosses an edge of $P$, say the edge $f \in P$.  We cannot have that $v_{2}$ is a top or bottommost dual vertex as $\mathcal{V}$ is not allowed to break a blue edge with a red edge that contains a top or bottommost vertex.  As such, there exist three edges, let us call them $g_{1},g_{2}$ and $g_{3}$, such that the set $\{f,g_{1},g_{2}, g_{3}\}$ forms a closed loop around the vertex $v_{2}$.  As $\mathcal{V}$ played a red edge that breaks a single blue edge we have that $\mathcal{H}$ is allowed to play $3$ edges in response.  As $v_{2}$ is not part of any component, we have that the three edges $\{g_{1},g_{2},g_{3}\}$ are not red edges, and so $\mathcal{H}$ claims all three of them.  These three edges, together with $P \setminus \{f\}$, form a path that, together with the bracket $B$, secures $C \cup \{v_{2}\}$.

Suppose next that $C$ is a bottom component, and that the dual edge $e$ lies across its gate $G$.  Then $e$ is of the form $e = (x,1.5)^{*}$, and the edges $(x+0.5,2)$ and $(x+1,1.5)$ are not red  (as otherwise $v_{2}$ would be part of some pre-existing component).  Thus the horizontal player $\mathcal{H}$ can claim these two edges and $C \cup \{v_{2}\}$ is now extra secure.  Similarly, if $C$ is a top component, and the dual edge $e$ lies across $G$ then, writing $e = (x,n- 0.5)^{*}$, we see that neither of the edges $(x+0.5,n-1)$ and $(x+1,n- 0.5)$ is red (as otherwise $v_{2}$ would be part of some pre-existing component).  Thus the horizontal player $\mathcal{H}$ can claim these two edges and $C \cup \{v_{2}\}$ is now extra secure. 

We now deal with the case where $C$ is a floating component and $e$ crosses an edge of its bracket $B$.  We divide here into sub-cases, depending on the type of the bracket $B$.  For each sub-case, there are some further sub-sub-cases to consider, depending on which edge of the bracket $B$ is crossed by $e$.

In all cases we will list the two edges $f_1$ and $f_2$ that constitute $\mathcal{H}$'s response, as well as a new bracket $B'$ or a new gate $G'$.  The blue path $P \cup \{f_{1},f_{2}\} \setminus \{e\}$ together with  $B'$ or $G'$ will then secure the new red component $C\cup\{v_2\}$. Since $v_2$ is not part of a pre-existing red component, it will follow that the two edges $f_1$ and $f_2$ are not red edges (so that $\mathcal{H}$ is free to claim them, or to turn them into blue double-edges if she had already claimed them in the past) and further that none of the edges in the new bracket $B'$ or gate $G'$ are red (so that $\mathcal{H}$'s move does indeed secure $C\cup\{v_2\}$, as claimed).

In our analysis, we will make use of Remark~\ref{remark: bracket symmetry} on the closure of the family of brackets under reflections swapping their corners, which will allow us to greatly reduce the number of cases we need to check.  Finally, before we dive into the case analysis, we would advise the reader to look at Figures \ref{Fig7} (Cases 2a, 2b), \ref{Fig8} (Case 2c), \ref{Fig9} (Cases 3a, 3b, 3c), and \ref{Fig10} (Cases 3d, 3e, 3f) in parallel with the proof, as the pictures there may greatly aid visualising the argument.

\medskip
\noindent
\textbf{Case 2a.}  The bracket $B$ is a bracket of Type $1$, with corner vertices $(x,y)$ and $(x+2,y+2)$ for some $x,y$.
\smallskip

Suppose $y = 1$.  If $e$ is the dual edge $(x+0.5,1)^{*}$, then $\mathcal{H}$ plays the two edges $(x+2,1.5)$ and $(x+2,2.5)$.  The new component $C \cup \{v_{2}\}$ is a bottom component that is extra secured by the path $P \cup \{(x+2,2.5)\}$ and the gate $G' = \{(x+2,1.5)\}$.  If $e$ is the dual edge $(x+1.5,1)^{*}$, then $\mathcal{H}$ plays the edges $(x+2,2.5)$ and $(x+0.5,1)$.  The new component $C \cup \{v_{2}\}$ is a bottom component that is secured by the path $P \cup \{(x+2,2.5),(x+0.5,1)\}$ and the gate $G' = \{(x+2,1.5)\}$.

Suppose instead $y \geqslant 2$. If $e$ is the dual edge $(x+0.5,y)^{*}$, then $\mathcal{H}$ plays the two edges $(x,y-0.5)$ and $(x+2,y+1.5)$.  The new bracket $B'$ is a bracket of Type $2$ with corner vertices $(x,y-1)$ and $(x+2,y+1)$. If $e$ is the dual edge $(x+1.5,y)^{*}$,  then $\mathcal{H}$ plays the two edges $(x+0.5,y)$ and $(x+2,y+1.5)$.  The new bracket $B'$ is a bracket of Type $3^{+}$ with corner vertices $(x+1,y)$ and $(x+2,y+1)$. Finally, if $e$ is the dual edge $(x+2,y+0.5)^{*}$ or the dual edge $(x+2,y+1.5)^{*}$, then 
we consider the dual edge $e$ and the bracket $B$ under the reflection that switches the corners of $B$ and determine our response by that given in the cases  $e = (x+0.5,y)^{*}$ or $(x+1.5,y)^{*}$, reflected back.  By Remark~\ref{remark: bracket symmetry}, the new bracket $B'$ thus obtained is a valid bracket.
 \begin{figure}[ht]
 	\centering
 	\includegraphics[scale=1]{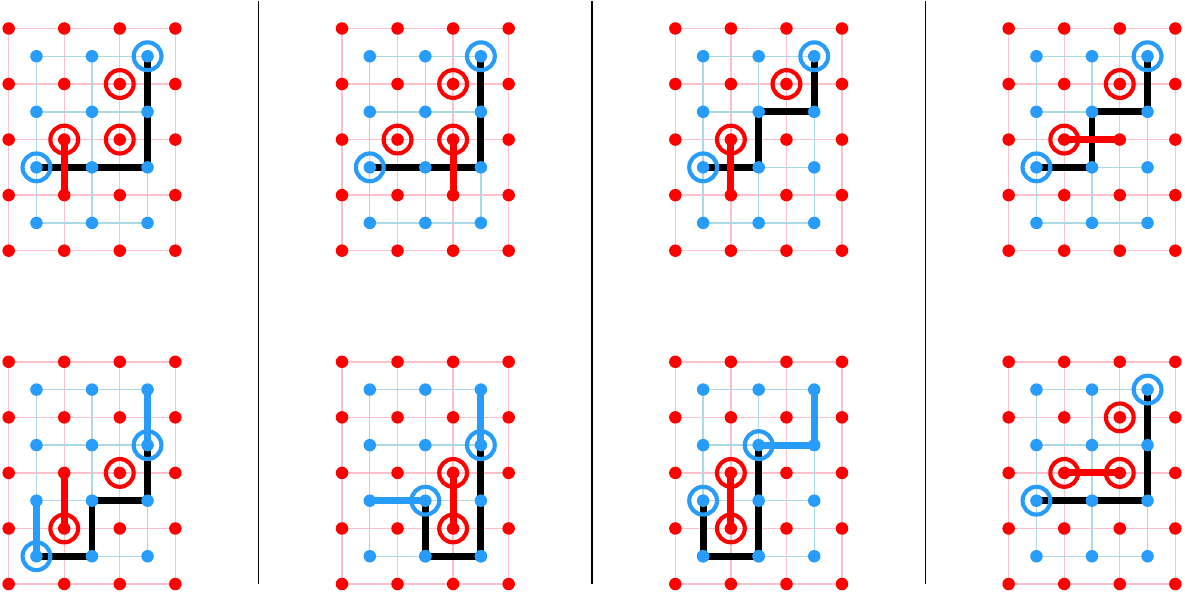}
 	\caption{The horizontal player $\mathcal{H}$'s response when the dual edge $e$ lies across the edge of a bracket of Type $1$ or $2$, as described in Cases $2$a and $2$b.  In each case the top picture shows the original bracket $B$, together with the dual edge $e$, while the picture below it shows the new bracket $B'$ and any newly claimed blue edges.}
 	\label{Fig7}
 \end{figure}

\medskip
\noindent
\textbf{Case 2b.}  The bracket $B$ is a bracket of Type $2$, with corner vertices $(x,y)$ and $(x+2,y+2)$ for some $x,y$.
\smallskip

As in Case 2a, as $B$ is fixed under the reflection that switches its corners, it is only necessary to deal with the cases $e = (x+0.5,y)^{*}$ and $e = (x+1,y+0.5)^{*}$.

Suppose first that $e$ is the dual edge $(x+0.5,y)^{*}$.  In this case $\mathcal{H}$ plays the two edges $(x+1.5,y+1)$ and $(x+2,y+1.5)$.  If $y = 1$, then the component $C \cup \{v_{2}\}$ is a bottom component secured by the path $P \cup \{x+2,2.5),(x+1.5,2)\}$ and the gate $G' = \{ (x+1,1.5)\}$.  If $y \geqslant 2$, then the new bracket $B'$ is a bracket of Type $3^{+}$ with corner vertices $(x,y)$ and $(x+1,y+1)$.

If instead $e$ is the dual edge $(x+1,y+0.5)^{*}$, then there is in fact no need for $\mathcal{H}$ to play any edges (so she may play them arbitrarily).  The new bracket $B'$ is a bracket of Type $1$ with corner vertices $(x,y)$ and $(x+2,y+2)$.

\medskip
\noindent
\textbf{Case 2c.}  The bracket $B$ is a bracket of Type $3^{+}$, with corner vertices $(x,y)$ and $(x+1,y+1)$ for some $x,y$.
\smallskip

Suppose first that $v_{2}$ is a bottommost dual vertex.  In this case, we have that $y = 2$ and $e = (x + 0.5,1)^{*}$.  The horizontal player $\mathcal{H}$ plays the edges $(x,1.5)$ and $(x+1,y+2.5)$.  The component $C \cup \{v_{2}\}$ is now a bottom component secured by the path $P \cup \{(x,1.5),(x+1,y+2.5)\}$ and the gate $G' = \{(x+1,y+1.5)\}$.

Suppose instead that $v_{2}$ is not a bottommost dual vertex.  If $e$ is the dual edge $(x,y-0.5)^{*}$, then $\mathcal{H}$ plays the two edges $(x-0.5,y)$ and $(x-1,y-0.5)$.  The new bracket $B'$ is a bracket of Type $1$ with corner vertices $(x-1,y-1)$ and $(x+1,y+1)$. If $e$ is the dual edge $(x+0.5,y-1)^{*}$, then $\mathcal{H}$ plays the two edges $(x,y-0.5)$ and $(x+1,y+0.5)$.  The new bracket $B'$ is a bracket of Type $3^{+}$ with corner vertices $(x,y-1)$ and $(x+1,y)$. If $e$ is the dual edge $(x+1,y-0.5)^{*}$, then $\mathcal{H}$ plays the two edges $(x,y-0.5)$ and $(x+1,y+0.5)$.  The new bracket $B'$ is a bracket of Type $3^{-}$ with corner vertices $(x,y-1)$ and $(x+1,y)$. Finally, if $e$ is the dual edge $(x+1,y+0.5)^{*}$, then $\mathcal{H}$ plays the two edges $(x,y-0.5)$ and $(x+1.5,y+1)$.  The new bracket $B'$ is a bracket of Type $2$ with corner vertices $(x,y-1)$ and $(x+2,y+1)$.

\begin{figure}[ht]
	\centering
	\includegraphics[scale=1]{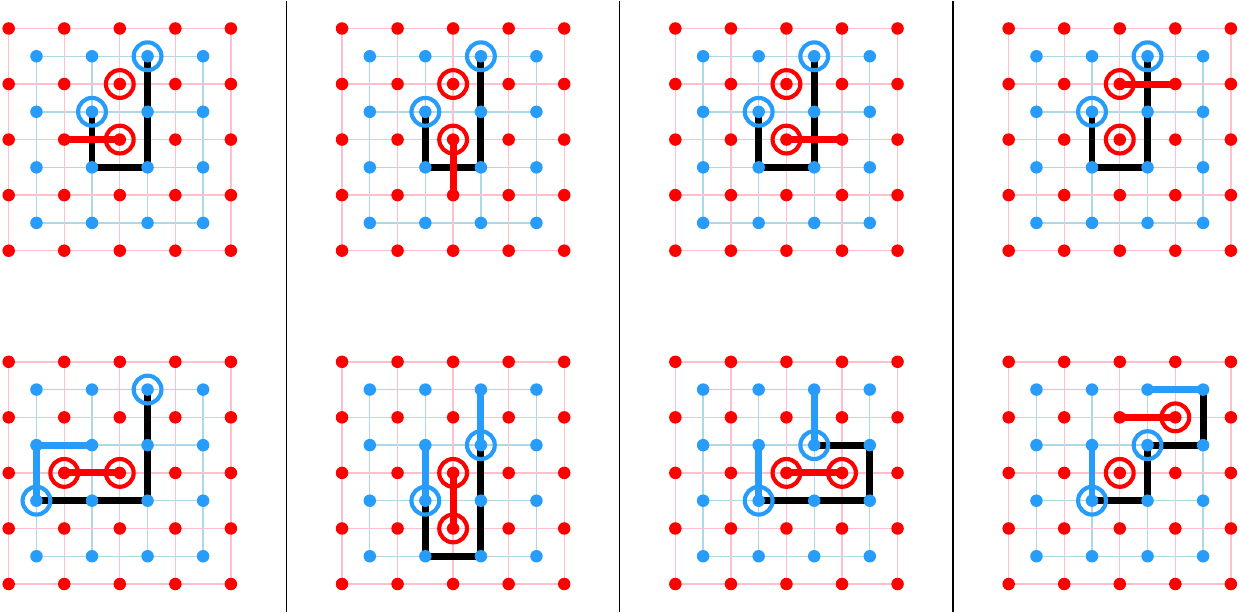}
	\caption{The horizontal player $\mathcal{H}$'s response when the dual edge $e$ lies across the edge of a bracket of Type $3^{+}$, as described in Case $2$c.  In each case the top picture shows the original bracket $B$, together with the dual edge $e$, while the picture below it shows the new bracket $B'$ and any newly claimed blue edges.}
	\label{Fig8}
\end{figure}

\medskip
\noindent
\textbf{Case 2d.}  The bracket $B$ is a bracket of Type $3^{-}$, with corner vertices $(x,y)$ and $(x+1,y+1)$ for some $x,y$.
\smallskip

Suppose first that $v_{2}$ is a bottommost dual vertex.  In this case we have that $y = 1$ and $e = (x + 0.5,1)^{*}$ or $(x + 1.5,1)^{*}$.  If $e = (x + 0.5,1)^{*}$, then $\mathcal{H}$ plays the edges $(x+1.5,2)$ and $(x+2,1.5)$.  The component $C \cup \{v_{2}\}$ is now a bottom component that is extra secured by the path $P \cup \{(x+1.5,2)\}$ and the gate $G = \{x+2,1.5\}$.  If $e = (x + 1.5,1)^{*}$, the horizontal player $\mathcal{H}$ plays the edges $(x + 0.5,1)$ and $(x + 1.5,2)$.  The component $C \cup \{v_{2}\}$ is now a bottom component that is secured by the path $P \cup \{(x + 0.5,1),(x + 1.5,2)\}$ and the gate $G' = \{x+2,1.5\}$.

We next suppose that $v_{2}$ is a topmost dual vertex.  We must have that $y = n-1$ and $e = (x+1.5,n)^{*}$.  The horizontal player $\mathcal{H}$ plays the edges $(x+0.5,n-1)$ and $(x+1.5,n-1)$.  The component $C \cup \{v_{2}\}$ is now a top component secured by the path $P \cup \{(x+0.5,n-1),(x+1.5,n-1)\}$ and the gate $G' = \{x+2,n-0.5\}$.

Finally, suppose that $v_{2}$ is not a bottommost or topmost vertex.  Then we consider the dual edge $e$ and the bracket $B$ under the reflection that switches the corners of $B$ and determine our response using Case $2$c (since $B$ is mapped to a bracket of Type $3^{+}$), reflected back. By Remark~\ref{remark: bracket symmetry}, the new bracket $B'$ thus obtained is a valid bracket.

\begin{case}\label{Case3}
Before $e$ is played, the vertex $v_{1}$ is part of some component $C_{1}$ while the vertex $v_{2}$ is part of some component $C_{2}$.
\end{case}

We first note that $C_{1}$ and $C_{2}$ cannot be the same component, as $\mathcal{V}$ cannot claim any dual edges that would create a closed cycle  of red dual edges (violating restriction (a) from the secure game definition).  For each $i = 1,2$, let $P_{i}$ and $B_{i}$ (or $G_{i}$) be the respective path and bracket (or gate) that makes $C_{i}$ a secure component.

We first deal with the case where $e^{*}\in P_{1}\cap P_{2}$.  Observe that $C_{1}$ and $C_2$ must then both be floating components. Indeed, suppose one of the components, say $C_{1}$, were a bottom component. Then $C_{2}$ cannot be a top or a floating component (else $\mathcal{V}$'s move would violate restriction (b) or (c)) and  further, $C_{2}$ cannot be a bottom component (else $\mathcal{V}$'s move would create a red dual arch, violating restriction (a)), a contradiction. Thus neither of $C_{1}$, $C_2$ can be a bottom component, and in a similar way neither of them can be a top component.

As $e^{*}\in P_{1}\cap P_{2}$ and the board was secure before $\mathcal{V}$'s turn, $e^{*}$ must have been a blue double-edge, and so $\mathcal{H}$ has $4$ edges to respond with. Thus $\mathcal{H}$ plays all the edges in $B_{1}$, of which there are at most $4$.  It is easy to see that the new component $C_{1} \cup C_{2}$ is secured by some path contained in the set of edges $(P_{1} \cup P_{2}\cup B_{1}) \setminus \{ e^{*}\}$ and the bracket $B_{2}$.

We next deal with the case where $e^{*}$ is an edge that lies in both $P_{1}$ and $B_{2}$ (or $G_{2}$).  By the same arguments as above (based on restrictions (a), (b) and (c)), we must have that $C_{2}$ is a floating component.  The horizontal player $\mathcal{H}$ has $3$ edges to respond with, and so she plays the three edges in $B_{2} \setminus \{e^{*}\}$.  Once again, it is easy to see that the new component $C_{1} \cup C_{2}$ is secured by some path contained in the set of edges $(P_{1} \cup P_{2}\cup B_{1}) \setminus \{ e^{*} \}$ and the bracket $B_{1}$ (or gate $G_{1}$).

Finally we need to deal with the case where $e^{*}$ is an edge that lies in both $B_{1}$ (or $G_{1}$) and $B_{2}$ (or $G_{2}$).  We first note that it is not possible for either $C_{1}$ or $C_{2}$ to be top components.  Indeed, if say $C_{1}$ was a top component, then $C_{2}$ must be a floating component (by restrictions (a) and (b)), yet there is no possible bracket for $C_{2}$ that can have an edge in common with $G_{1}$.  Next, let us suppose that $C_{1}$ is a bottom component whose gate $G_{1}$ consists of the edge $(x,1.5)$.  By restrictions (a) and (b), $C_{2}$ is a floating component and $B_{2}$ must be a bracket of type $3^{+}$ with corners $(x,2)$ and $(x+1,3)$ (no other bracket type is compatible with $G_{1}$).  The horizontal player $\mathcal{H}$ then plays the edges $(x+1,1.5)$ and $(x+1,2.5)$.  The new component $C_{1} \cup C_{2}$ is a bottom component extra-secured by a path contained in the set of edges $P_{1} \cup P_{2} \cup \{(x+1,2.5)\}$ and the gate $\{(x+1,1.5)\}$.

Finally, if $C_{1}$ and $C_{2}$ are both floating components, then we have to split into sub-cases, depending on the bracket-type of $B_{1}$ and $B_{2}$, and on which edge they share.  Note to begin with that if $B_1$, $B_2$ are both of Type $1$ or $2$ and have an edge in common, then they must share an interior vertex, contradicting the fact that $C_1$ and $C_2$ are distinct components. Thus without loss of generality, we may assume that $B_{1}$ is a bracket of Type $3^{+}$ or $3^{-}$.  We deal below with the case where $B_{1}$ is a bracket of Type $3^{+}$ with corner vertices $(x,y)$ and $(x+1,y+1)$.  The case where $B_{1}$ is a bracket of Type $3^{-}$ will then follow by considering the reflection switching $B_{1}$'s two corner vertices and making use of Remark~\ref{remark: bracket symmetry}.

In each of the following sub-cases we will list the set $P_3$ of two (or fewer) edges from $B_{1}\cup B_{2}\setminus\{e^*\}$ that $\mathcal{H}$ plays and the location and type of a new bracket $B$.   It is easy to check then that there is a blue path $P$ contained within the edges of $P_{1} \cup P_{2}\cup P_3$ such that $P$ and $B$ together secure the new component $C_{1} \cup C_{2}$.  In the sub-cases below, we cover all ways in which a bracket of Type $3^{+}$ and another bracket could share an edge. In each sub-case (except Case $3$g, which we deal with via a reflection and Remark~\ref{remark: bracket symmetry}), we let $e=(x, y-0.5)^*$ be the dual edge played by $\mathcal{V}$.
 \begin{figure}[ht]
 	\centering
 	\includegraphics[scale=1]{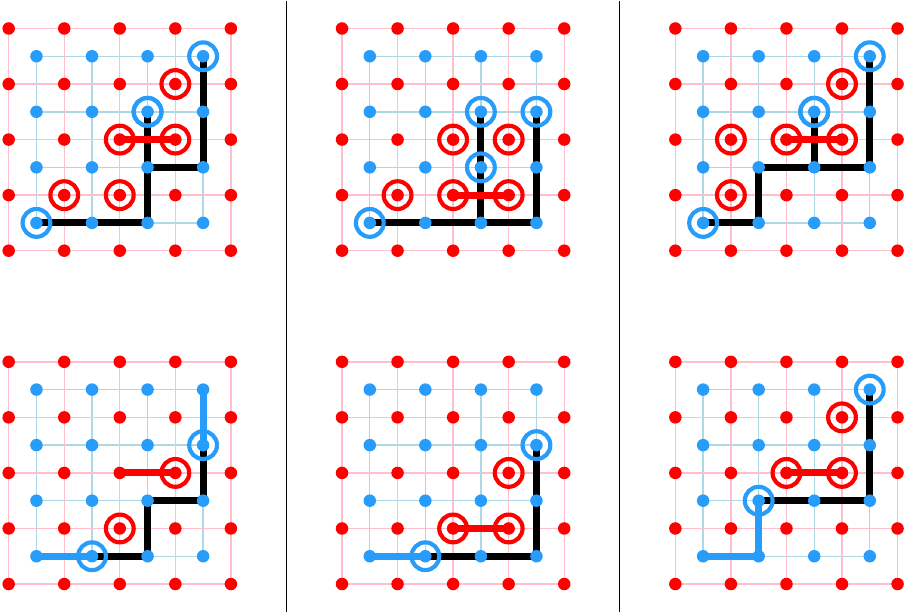}
 	\caption{The horizontal player $\mathcal{H}$'s response when the dual edge $e$ lies across  an edge of two brackets, $B_{1}$ and $B_{2}$, as described in Cases $3$a, $3$b and $3$c.  In each case the top picture shows the original brackets, together with the dual edge $e$, while the picture below it shows the new bracket $B$ and any newly claimed blue edges.}
 	\label{Fig9}
 \end{figure}

\medskip
\noindent
\textbf{Case 3a.}  The bracket $B_{2}$ is a bracket of Type $1$, with corner vertices $(x-2,y-2)$ and $(x,y)$.
\smallskip
In this case $\mathcal{H}$ plays the edges $(x-1.5,y-2)$ and $(x+1,y+0.5)$.  The bracket of $B$ is a bracket of Type $2$ with corner vertices $(x-1,y-2)$ and $(x+1,y)$.

\medskip
\noindent
\textbf{Case 3b.}  The bracket $B_{2}$ is a bracket of Type $1$, with corner vertices $(x-2,y-1)$ and $(x,y+1)$.
\smallskip
In this case $\mathcal{H}$ plays the edge $(x-1.5,y-1)$.  The bracket of $B$ is a bracket of Type $1$ with corner vertices $(x-1,y-1)$ and $(x+1,y+1)$.

\medskip
\noindent
\textbf{Case 3c.}  The bracket $B_{2}$ is a bracket of Type $2$, with corner vertices $(x-2,y-2)$ and $(x,y)$.
\smallskip
In this case $\mathcal{H}$ plays the edges $(x-1.5,y-2)$ and $(x-1,y-1.5)$.  The bracket of $B$ is a bracket of Type $1$ with corner vertices $(x-1,y-1)$ and $(x+1,y+1)$.
 \begin{figure}[ht]
 	\centering
 	\includegraphics[scale=1]{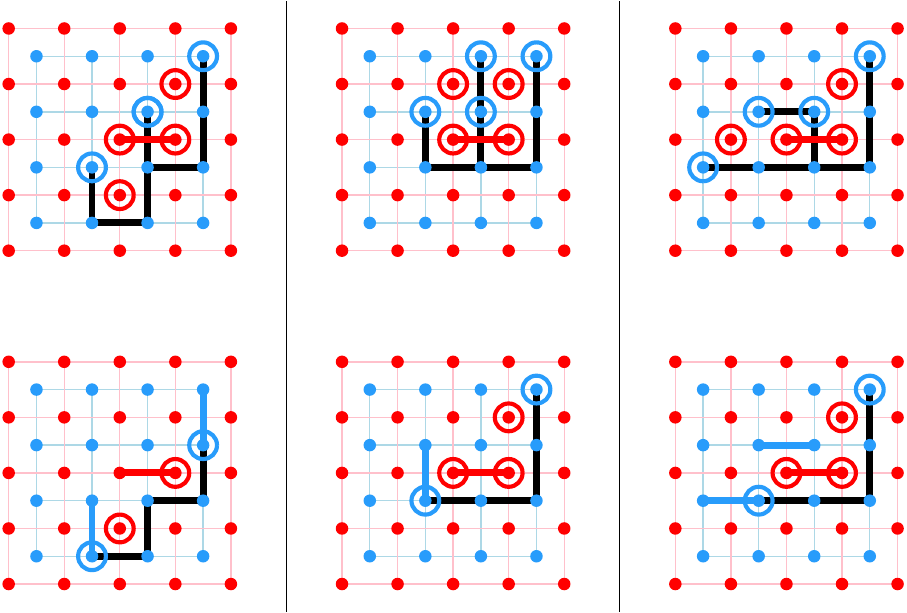}
 	\caption{The horizontal player $\mathcal{H}$'s response when the dual edge $e$ lies across  an edge of two brackets, $B_{1}$ and $B_{2}$, as described in Cases $3$d, $3$e and $3$f.  In each case the top picture shows the original brackets, together with the dual edge $e$, while the picture below it shows the new bracket $B$ and any newly claimed blue edges..}
 	\label{Fig10}
 \end{figure}
 
\medskip
\noindent
\textbf{Case 3d.}  The bracket $B_{2}$ is a bracket of Type $3^{+}$, with corner vertices $(x-1,y-1)$ and $(x,y)$.
\smallskip
In this case $\mathcal{H}$ plays the edges $(x-1,y-1.5)$ and $(x+1,y+0.5)$.  The bracket of $B$ is a bracket of Type $2$ with corner vertices $(x-1,y-2)$ and $(x+1,y)$.

\medskip
\noindent
\textbf{Case 3e.}  The bracket $B_{2}$ is a bracket of Type $3^{+}$, with corner vertices $(x-1,y)$ and $(x,y+1)$.
\smallskip
In this case $\mathcal{H}$ plays the edge $(x-1,y-0.5)$.  The bracket of $B$ is a bracket of Type $1$ with corner vertices $(x-1,y-1)$ and $(x+1,y+1)$.

\medskip
\noindent
\textbf{Case 3f.}  The bracket $B_{2}$ is a bracket of Type $3^{-}$, with corner vertices $(x-2,y-1)$ and $(x-1,y)$.
\smallskip
In this case $\mathcal{H}$ plays the edges $(x-1.5,y-1)$ and $(x-0.5,y)$.  The bracket of $B$ is a bracket of Type $1$ with corner vertices $(x-1,y-1)$ and $(x+1,y+1)$.

\medskip
\noindent
\textbf{Case 3g.}  The bracket $B_{2}$ is a bracket of Type $3^{-}$, with corner vertices $(x-1,y-2)$ and $(x,y-1)$, and $e$ is the dual edge $(x+0.5, y-1)^*$.
\smallskip
This case is in fact already dealt with, as the situation is identical to the previous case up to the reflection switching the corners of $B_{2}$.
\medskip

\noindent With Cases $3$a-g above, we have covered all possible cases and shown that $\mathcal{H}$ has a winning strategy for the secure game.
\end{proof}
\noindent With a winning strategy for the secure game in hand, we now show that in the  $q$-double-response game $\mathcal{H}$ can ensure that the grid is secure at the end of each of her turns. 

\begin{proof}[Proof of Theorem~\ref{theorem: double-response}]
Suppose the grid is secure and let $D$ be the set of dual edges claimed by $\mathcal{V}$ on his turn, where $|D| = r \leqslant q$.  The horizontal player $\mathcal{H}$ begins by picking a judicious ordering (to be specified later) of the elements of $D$ as $\{e_1, e_2, \ldots, e_r\}$ , and then proceeding as if she was playing the secure game, pretending that $\mathcal{V}$ plays $e_1$, $e_2$, $\ldots$, $e_r$ in that order and responding to each $e_i$ in turn.

For each $i = 1,\ldots, r$, let $L_{i}$ be the set of blue edges (including blue double-edges) that $\mathcal{H}$ has claimed after $i$ of her turns have occurred in this auxiliary secure game. We do not include in $L_{i}$ any blue edge broken by $\mathcal{V}$ in any of his first $i$ turns.

Recall that in the secure game, $\mathcal{H}$ responds to $\mathcal{V}$'s claim of the dual edge $e_i$ with $2$, $3$ or $4$ edges, depending on whether $e_i$ breaks $0$, $1$ or $2$ blue edges.
 As such, we have that $|L_{i}| \leqslant 2i$ for all $i = 1,\ldots,r$.  Once $\mathcal{H}$ has gone through every dual edge of $D$ she has a set $L_{r}$ of at most $2r$ edges such that if $\mathcal{H}$ claims all the edges in $L_{r}$ in response to $\mathcal{V}$'s claim of $D$, then the grid is back to a secure position in the $q$-double/response game.  The only problem that could occur in this scenario is that during some turn of the auxiliary secure game, say turn $i$, the dual edge $e_i$ claimed by $\mathcal{V}$ breaks one of the restrictions (a)--(c) we imposed on the secure game.  We show below that this can be avoided by picking a judicious ordering on $D$.  Combined with Lemma \ref{Lemma2}, this will complete the proof Theorem~\ref{theorem: double-response}.

The first restriction (a) on $\mathcal{V}$'s moves in the secure game is that $\mathcal{V}$ is not allowed to claim a dual edge as red if doing so would create a cycle or an arch of red dual edges.  As we have shown one can assume $\mathcal{V}$ never plays such an edge in the $q$-double-response game, this restriction will not be broken by any of the dual edges in $D$.

The second restriction (b) is that $\mathcal{V}$ may not claim a dual edge if doing so connects a top component to a bottom component.  We know by Lemma \ref{Lemma1} that if the grid is secure at the beginning of a turn of the $q$-double-response game, then $\mathcal{V}$ cannot win in that turn.  As such, there cannot be a dual edge in $D$ that connects a top component to a bottom component, and so this restriction is not broken either.

The third and final restriction (c) is that if $C$ is a floating component and $P$ is the path of blue edges that helps secure $C$, then $\mathcal{V}$ may not claim a red dual edge that breaks a blue edge from $P$ if claiming that dual edge would turn $C$ into either a bottom or top component. It is here that our judicious ordering of $D$ comes into play and ensures restriction (c) is respected.

We order the dual edges in $D$ as follows.  Due to restrictions (a) and (b), every top (respectively bottom) component is a rooted tree whose root is a topmost (respectively bottommost) dual vertex.  Let $D$ be ordered in any way such that, if $e, e' \in D$ are two dual edges that are part of the same bottom or top component $C$ and $e$ is strictly closer in graph distance in $C$ to the root of $C$ than $e'$, then $e$ appears before $e'$ in the ordering of $D$.  (Such an ordering clearly exists, by proceeding component by component.) We claim that ordering $D$ in this way guarantees that $\mathcal{V}$ never breaks the third restriction when we play the dual edges one by one.  Indeed, suppose there is a dual edge $e_i \in D$ such that before $e_i$ is played in the secure game, there exists a floating component $C$, secured by a path $P$ and a bracket $B$,  that becomes a bottom or top component once $e_i$ has been played.  Given our ordering on $D$, no other edge of $D$ meeting $C$ can have been played in the secure game before $e_i$. In particular, all the edges of $P$ were present before $\mathcal{V}$ played the dual edge-set $D$ in the $q$-double-response game and $\mathcal{H}$ introduced the auxiliary secure game. In particular $e_i$ cannot break an edge of $P$ (i.e $e$ must lie across $B$), and restriction (c) is respected.
\end{proof}
\begin{remark} As pointed out by a referee, our proof of Theorem~\ref{theorem: (2q,q)-game} implies that if  $p>2q$ then when Maker plays a $(p, q)$-crossing game on $S_{\infty \times (q+1)}$, she can not only ensure that Breaker never claims a set of edges corresponding to a top-bottom dual crossing path, but in addition, using her $p-2q$ extra edges at each turn, she can actually build an unbounded component. In this sense, $p>2q$ is a strong Maker win. On the other hand, this is not true if $p=2q$. Indeed, Breaker can follow the strategy of claiming on each of his turns the bottom $q$ edges of a top-bottom dual crossing path that lies at graph distance at least $2q$ from any previously claimed edge. Then should Maker fail to respond by playing all her $2q$ edges within distance at most $q$ from Breaker's edges, Breaker is able to complete a top-bottom dual crossing path on his next turn. Thus at $p=2q$ on $S_{\infty \times (q+1)}$, Maker has only just enough power to stymie Breaker, but is not able to actively construct a large structure for herself. \end{remark}

\section{Other graphs and other games}\label{section: other graphs and other games}
The crossing games we study in this paper may be viewed as special cases of the following generalisation of the Shannon switching game.
\begin{definition}[$(p,q)$-Shannon switching game]\label{definition: (p,q)-Shannon switching game}
A Shannon game-triple is a triple $(G,A,B)$, where $G$ is a finite multigraph (possibly with loops) and $A,B$ are sets of vertices from $G$. For $p,q\in \mathbb{N}$, the $(p,q)$-Shannon switching game on $(G,A,B)$ is played on the board $E(G)$ as follows.

Two players, Maker and Breaker, play in alternating turns. Maker plays first and in each of her turns claims $p$ (as-yet-unclaimed) edges of the board $E(G)$; Breaker in each of his turns answers by claiming $q$ (as-yet-unclaimed) edges of the board. Maker wins the game if she manages to claim all the edges of a path joining $A$ to $B$ (i.e. a path from some $a\in A$ to some $b\in B$ --- we call such a path an \emph{$A$--$B$ crossing path}). Otherwise, Breaker wins.	
\end{definition}
The $(p,q)$-crossing games we study in this paper are instances of the  $(p,q)$-Shannon switching game on $(G,A,B)$, where $G=S_{m \times n}$ and $A$ and $B$ are the sets of left-hand side and right-hand side vertices of $S_{m \times n}$ respectively. 

The generalised Shannon switching game satisfies some obvious monotonicity properties with regard to the board, which we record in Proposition~\ref{proposition: monotonicity for Maker/Breaker} below.  Given a multigraph $G$ and two distinct vertices $u,v \in V(G)$, let $m_{G}(u,v)$ denote the number of edges between $u$ and $v$, and let $m_{G}(v)$ denote the number of loops at $v$.  Let $G'$ be any multigraph obtained by taking $G$, deleting some vertex $v \in V(G)$, replacing it with two adjacent vertices $v_{1},v_{2}$, and then adding in edges adjacent to $v_{1}$ or $v_{2}$ until the relations
\begin{equation}
m_{G'}(v_{1},v_{2}) + m_{G'}(v_{1}) + m_{G'}(v_{2}) = m_{G}(v) + 1, \nonumber
\end{equation}
and
\begin{equation}
m_{G'}(u,v_{1})+m_{G'}(u,v_{2}) = m_{G}(u,v), \nonumber
\end{equation}
are satisfied for all $u \in V(G)\setminus \{v\}$.  We refer to this process as \textit{vertex-separation}.  Vertex separation may be thought of as an inverse operation to performing an \textit{edge-contraction} of the edge $\{v_{1},v_{2}\}$ in $G'$.

\begin{proposition}\label{proposition: monotonicity for Maker/Breaker}
Let $(G, A, B)$ and $(G',A',B')$ be Shannon game-triples. Suppose $(G',A',B')$ may be obtained from $(G,A,B)$ by a sequence of vertex-deletions, edge-deletions and vertex-separations.

Then the following hold:
\begin{enumerate}[(i)]
	\item if Maker has a winning strategy for the $(p,q)$-Shannon switching game on $(G',A',B')$, then Maker also has a winning strategy for the $(p,q)$-Shannon switching game on  $(G,A,B)$;
	\item if Breaker has a winning strategy for the $(p,q)$-Shannon switching game on $(G,A,B)$, then Breaker also has a winning strategy for the $(p,q)$-Shannon switching game on  $(G',A',B')$. \quad 	\qedsymbol
	\end{enumerate}

\end{proposition}

One can build a natural rooted tree structure on game-triples: we say that a generalised Shannon-switching game is played according to the \emph{Breaker First rule} (BF rule) if Breaker is allowed to make the first move instead of Maker. Then starting with a game triple $(G,A,B)$, we build a tree by letting $(G,A,B)$ be the root. The children of $(G,A,B)$ are all game-triples $(G',A',B')$ with BF rule obtained from $(G,A,B)$ by contracting $p$ edges. The children of a game triple $(G', A', B')$ with BF rule are then all game-triples $(G'',A'', B'')$ obtained by deleting $q$ edges. Repeating this operation, we build a game tree, whose leaves will consist of game-triples $(G_l, A_l,B_l)$ with $A_{l}\cap B_l\neq \emptyset$ (Maker's win) or with no $A_l$--$B_l$ paths (Breaker's win).
Ultimately, the Maker's win leaves are equivalent under optimal play to the game $(K_1, \{1\}, \{1\})$ played on the one-vertex graph $K_1$, while Breaker's leaves are equivalent to the game $((K_2)^c, \{1\}, \{2\})$ played on the two-vertex non-edge graph $K_2^c$.

By considering Proposition~\ref{proposition: monotonicity for Maker/Breaker} and the game-tree described above, our results on $(p,q)$-crossing games give winning strategies for a number of related Shannon-switching games. In a slightly different direction, our winning strategy for the $(2q-1, q)$-crossing game on sufficiently long strips is easily adapted to a much more general setting.  
\begin{theorem}\label{theorem: general strip theorem}
	Let $(G,A,B)$ be a Shannon game-triple.  Assume that 
	\begin{enumerate}[(i)]
		\item the game-board $E(G)$ may be split up into $m$ edge-disjoint strips $S_1, \ldots, S_m$;
		\item on each strip $S_i$ we have a local game $(G_i, A_i,B_i)$ such that if Breaker wins that game, then Breaker wins the global game on $(G,A,B)$;
		\item for each $i$, Breaker has a winning strategy for the $(p,q)$-Shannon switching game on $(G_i,A_i, B_i)$ under BF rules that ensures Breaker's victory in at most $T$ turns.
	\end{enumerate}
	Then provided
	\begin{equation}
		m \geqslant (s(p+1))^{T} + l(p+1)-1 \nonumber
		\end{equation}
	Breaker has a winning strategy for the $(l(p+1)-1 , lq)$-Shannon switching game on $(G,A,B)$, where $s = l(p+q+1)-1$.
\end{theorem}
\begin{proof}
We generalise the arguments of Theorem \ref{theorem: (2q-1,q)-game} as follows.  At any point in the game, for $0 \leqslant j \leqslant p$ we say a strip $S_{i}$ is $(k,j)$\textit{-valid} if it contains exactly $kq$ red edges and is in a winning position for Breaker in the $(p,q)$-Shannon switching game on $(G_{i},A_{i},B_{i})$, with Maker getting to play any $j$ edges first before the game resumes with it being Breaker's turn to play.  If a strip is not $(k,j)$-valid for any $0 \leqslant j \leqslant p$, then we say that it is \textit{invalid}.  As Breaker has a winning strategy for the $(p,q)$-Shannon switching game on $(G_{i},A_{i},B_{i})$ under BF rules for each $i$, we have that each strip starts as $(0,0)$-valid.  Note that for $j' \geqslant j$, we have that any strip that is $(k,j')$-valid is also $(k,j)$-valid.  Thus, we say that a strip is \textit{exactly} $(k,j)$-valid if it is $(k,j)$-valid but not $(k,j+1)$-valid.  Note that if any strip $S_{i}$ is $(k,j)$-valid, then Breaker can play $q$ edges in $S_{i}$ and turn it in into a $(k+1, p)$-valid strip.  Indeed, as $S_{i}$ is $(k,j)$-valid, we know that it is also $(k,0)$-valid, and so it is in a winning position for Breaker in the $(p,q)$-Shannon switching game where it is Breaker's turn to play.  Breaker plays the $q$ edges that a winning strategy would prescribe, so that the strip is in a winning position for Breaker in the $(p,q)$-Shannon switching game, even though it is Maker's turn to play.  Thus Breaker has turned $S_{i}$ into a $(k+1,p)$-valid strip.

The game begins with Maker playing edges in up to $l(p+1)-1$ different strips, possibly making them invalid.  From here we split the game in to a number of different phases.  We will show by induction on $k$ that for each $k = 0,1,\ldots,T$, at the start of phase $k$ it will be Breaker's turn to play and the number of $(k,0)$-valid strips will be at least $(s(p+1))^{T-k}$.  As noted above, our inductive statement is clear when $k = 0$.  Suppose the statement is true for $k$.  On each turn in phase $k$, Breaker will choose $l$ different $(k,0)$-valid strips and play $q$ edges in each, turning them into $(k+1,p)$-valid strips.  Maker can now distribute their $l(p+1)-1$ among all the strips as she likes.  In the worst case scenario, each edge that Maker plays can either turn a strip that is exactly $(k+1,j)$-valid  into one that is exactly $(k+1,j-1)$-valid (when $j \geqslant 1$), or turn a $(k,0)$-valid or $(k+1,0)$-valid strip into an invalid one.  For each $j = 0,1,\ldots,p$, let $R_{t}(j)=R_t(j,k)$ be the number of exactly $(k+1,j)$-valid strips on the board after a total of $t$ combined edges in round $k$ have been played by the two players.  Moreover, let $R_{t} =R_t(k)$ be given by $R_t= \sum_{j = 0}^{p}(j+1)R_{t}(j)$.

We have that, if after $t$ edges have been played it is Breaker's turn to play and he plays $q$ edges, then $R_{t+q} = R_{t} + p+1$.  On the other hand, if after $t$ turns it is Maker's turn to play, then we have that $R_{t+1} \geqslant R_{t}-1$.  As Breaker plays a total of $lq$ edges while Maker plays a total of $l(p+1)-1$ edges on their respective turns for a combined total of $s$ edges, we have that $R_{rs} \geqslant r$ for all $r \in \mathbb{Z}_{\geqslant 0}$, at least until phase $k$ ends.  Breaker decides that phase $k$ has finished and phase $k+1$ has begun when $R_{rs} \geqslant (p+1)((p+1)s)^{T-k-1}$ for some $r \in \mathbb{Z}_{\geqslant 0}$.  Note that after Maker and Breaker have both completed their turns, the number of $(k,0)$-valid strips has decreased by at most $s$. Thus, as the number of $(k,0)$-valid strips at the start of phase $k$ is at least $(s(p+1))^{T-k}$, we know that the number of $(k,0)$-valid strips for Breaker to play in will not run out before $R_{rs} \geqslant (p+1)((p+1)s)^{T-k-1}$.  As $R_{rs} \geqslant (p+1)((p+1)s)^{T-k-1}$ we have that the number of $(k+1,0)$-valid strips at the start of phase $k+1$ is at least $((p+1)s)^{T-k-1}$ and it is Breaker's turn to play, as required.

To finish the proof, we note that at the start of phase $T$, there is at least one strip $S_{i}$ that is $(T,0)$-valid, and has been obtained by Breaker following a winning strategy on this strip for the $(p,q)$-Shannon switching game under the BF rules.  As Breaker can win the $(p,q)$-Shannon switching game on $S_{i}$ in at most $T$ moves, he has in fact won the local $(p,q)$-Shannon switching game on this strip $S_{i}$, and with it the global $(l(p+1)-1 , lq)$-Shannon switching game on $(G,A,B)$.
\end{proof}
\begin{remark}
The bound on $m$ given in Theorem~\ref{theorem: general strip theorem} is precisely the bound on the number of strips $m_0/(n+1)$ given in the proof of Theorem~\ref{theorem: (2q-1,q)-game} --- simply substitute in the values $p=q=1$ corresponding to the powers of the players in the local game on the strips and replace $l$ by $q$ in the bound for $m$ to recover the bound on $m_0/(n+1)$.
Note in particular that $q$ plays a different role in statements of the two theorems.
\end{remark}

Just as we have been able to generalise our winning Breaker strategy for the $(2q-1,q)$-crossing game to other Shannon game-triples, we believe our winning Maker strategy for the $(2q,q)$-crossing game on $S_{m \times (q+1)}$, as described in the proof of Theorem \ref{theorem: (2q,q)-game}, can be adapted to a number of other planar lattices. The key idea here is that if $\Gamma$ is a planar lattice where an isoperimetric  inequality similar to that of Lemma \ref{lemma: isoperimetric lemma} holds, then a Maker strategy similar to that in the proof of Theorem \ref{theorem: (2q,q)-game} should work in $\Gamma$.  More precisely, suppose $\Gamma$ is a planar lattice such that there exist a constant $a$ such that for all $k \in \mathbb{N}$ and for all connected components $C$ comprised of $k$ edges, the dual boundary cycle to $C$ consists of at most $ak + (a+2)$ edges\footnote{The quantity $ak + (a+2)$ comes from considering how $\mathcal{H}$'s strategy for the secure game, as described in the proof of Lemma \ref{Lemma2}, might adapt to other lattices.}.  In this case, we believe that there exists some constant $c$ such that Maker has a winning strategy for $(aq,q)$-crossing games on all arbitrarily long substrips of  $\Gamma$ of ``width'' at least $c$.  Of course, modifying our proof of Theorem~\ref{theorem: (2q,q)-game} to adapt it to a given planar lattice $\Gamma$ will require a careful definition of brackets and a large amount of case-checking (as is already the case in the proof of Theorem \ref{theorem: (2q,q)-game} itself), and so we make no attempt to do so here.


Given our original motivation from percolation theory, it would be natural to study Shannon-switching games on strips of any of the standard $2$-dimensional lattices studied in percolation. For instance, who wins crossing games on `rectangular-shaped' subgraphs of the triangular, honeycomb or Kagome lattices? More generally, this is a natural problem for any of the $11$ Archimedean lattices. 

In a different direction, one could consider \emph{site-percolation} rather than \emph{bond-percolation}, by playing variants of our generalised Shannon switching games where the players take turns claiming vertices rather than edges. One famous example of such a game is the game of Hex, where the players take turn claiming vertices on a subset of the triangular lattice, both trying to create certain crossing paths.  It is easy to prove that a vertex-analogue of Lemma~\ref{lemma: isoperimetric lemma} holds in this lattice ---  i.e. that any set of $k$ vertices inducing a connected component of the triangular lattice can be surrounded by a bounding cycle consisting of at most $2k+4$ vertices --- and we guess that our Maker winning strategy for the $(2q,q)$-crossing game should carry over without excessive technicalities (but not without care and case-checking).

\section{Concluding remarks}\label{section: concluding remarks}
There are many questions arising from our work. Outside of the special cases $(p,q)=(1,1)$, $p\geqslant 2q$ and $p\leqslant \frac{q}{2}$, the problem of determining which of Maker or Breaker has a winning strategy for the $(p,q)$-crossing game  on $S_{m\times n}$
is completely open for pairs $(m,n)$ that fall outside the scope of Theorem~\ref{theorem: (2q-1,q)-game} and Propositions~\ref{proposition: (p,p)-game}-\ref{proposition: (p, p+5r)-game}. Resolving this seems an obvious (but challenging) problem.
\begin{problem}\label{problem: (p,q)-crossing game}
	Given natural numbers $p,q,n\in \mathbb{N}$, determine the greatest $m \in \mathbb{N}$ such that Maker has a winning strategy for the $(p,q)$-crossing game on $S_{m \times n}$.
\end{problem}
As a special, easier case, one could consider the following problem of determining the optimal value of $m$ 
in the variant of the $(1,1)$-crossing game where Maker gets an extra edge every $M$ turns.
\begin{question}\label{question: extra power}
	Suppose we play a variant of the $(1,1)$-crossing game where every $M$ turns Maker gets to claim an extra edge. Given $n, M\in \mathbb{N}$, what is the greatest $m$ such that Maker has a winning strategy when playing on $S_{m \times n}$?
\end{question}
It is not hard to show Maker can win in this variant for some $m=n+ \Omega(\log n)$, and it would be very interesting to determine whether she has a winning strategy for $m=\lfloor (1+\varepsilon)n\rfloor$ for some constant $\varepsilon=\varepsilon(M)>0$.

In a similar spirit, setting one's sights slightly lower than Problem~\ref{problem: (p,q)-crossing game}, one could try to prove that having extra power allows one to win on a significantly longer board.
\begin{conjecture}\label{conjecture: extra power means an epsilon longer board}
	The following hold:
\begin{enumerate}[(i)]
	\item for every $q\in \mathbb{N}$, there exists $\varepsilon >0$ such that for all $n$ sufficiently large, Maker wins the $(q+1, q)$ game on $S_{\lceil (1+\varepsilon)n\rceil \times n}$;
	 \item for every $p\in \mathbb{N}$, there exists $\varepsilon >0$ such that for all $m$ sufficiently large, Breaker wins the $(p, p+1)$ game on $S_{m \times \lceil (1+\varepsilon)m\rceil}$.
\end{enumerate}
\end{conjecture}
An even more basic problem is showing that when the powers are balanced, Breaker should win on a narrower board, overcoming Maker's first-player advantage.
\begin{conjecture}\label{conjecture: at equal power Breaker wins on an epsilon -narrowerboard}
For every $\varepsilon >0$ and every $p\in \mathbb{N}$, there exists  $m_0\in \mathbb{N}$ such that for all $m\geqslant m_0$, Breaker wins the $(p,p)$-crossing game on $S_{m \times \lceil (1-\varepsilon)m\rceil}$.
\end{conjecture}

In a different direction, one may ask for optimal bounds on $m$ in Theorem~\ref{theorem: (2q-1,q)-game}. 
\begin{question}\label{question: optimal m for (2q-1,q)}
Let $n,q \in \mathbb{N}$. What is the smallest $m_0=m_0(n,q)$ such that Breaker wins the $(2q-1,q)$-crossing game on $S_{m_0 \times n}$? In particular, for $q$ fixed, is $m_0(n,q)$ subexponential in $n$?
\end{question}
A related question, which would help improve the bounds on $m$ for the Breaker strategy we developed in the proof of Theorem~\ref{theorem: (2q-1,q)-game} is the following:
\begin{question}\label{question: length of Bridg'it}
Under perfect play, how long does a game of Bridg-it last?
\end{question}
\noindent We make no attempt to answer this question here, however we believe that it may be possible to shed some light on the answer through careful analysis of the Maker-win strategy recorded in Theorem~\ref{theorem: Maker winds Bridg-it with any first move}.

In yet another direction, efforts to apply the biased Erd{\H o}s--Selfridge~\cite{Beck82, ErdosSelfridge73} criterion to Problem~\ref{problem: (p,q)-crossing game} leads to some intriguing questions on weighted sums over crossing paths, connected to the study of fugacity in statistical physics and to problems in analytic combinatorics (see e.g.~\cite{BousquetGuttmannJensen05}). Explicitly, let $\mathcal{H}(m,n)$ denote the collection of all left-to-right crossing paths in the rectangle $S_{m \times n}$.  Given a path $\pi \in \mathcal{H}(m,n)$, let $\ell(\pi)$ denote its length. Then the biased Erd{\H o}s--Selfridge criterion due to Beck implies that if 
\begin{align}\label{inequality: biased Erdos--Selfridge criterion}
\sum_{\pi \in \mathcal{H}} (1+q)^{-\frac{\ell(\pi)}{p}} < \frac{1}{1+q}, 
\end{align}
then Breaker has a winning strategy for the $(p,q)$-crossing game on $S_{m \times n}$.  In particular, suppose $m=\rho n$ for some $\rho>0$, and that we knew that, as $n\rightarrow \infty$, the number of crossing paths of $S_{m \times n}$ of length $\ell$ grew no faster than $(\lambda_{\rho}+o(1))^{\ell}$, for some $\rho$-dependent constant $\lambda_{\rho}$ (this would be a ``crossing path'' analogue of the connective constant familiar from the study of self-avoiding walks). Then (\ref{inequality: biased Erdos--Selfridge criterion}) would imply that Breaker has a winning strategy whenever $\lambda_{\rho} < (1+q)^{\frac{1}{p}}$. If for some $\rho<3$ the value of $\lambda_{\rho}$ were found to be sufficiently small so that $\lambda_{\rho} < 2$, this would show that Breaker wins the $(2,3)$-crossing game on $S_{m \times n}$ for all $n$ sufficiently large, giving a non-trivial improvement on what we know about that game. Of especial interest would be the case $\rho=1$ --- one would guess that Breaker's extra power in the $(2,3)$-crossing game would allow him to win on $S_{n \times n}$, say, but we have currently no proof of even this weakening of Conjecture~\ref{conjecture: extra power means an epsilon longer board}(ii).

Finally, variants of our games on other lattices or where the players claim vertices rather than edges, as discussed in Section~\ref{section: other graphs and other games}, are both interesting and almost completely open.
\section*{Acknowledgements}
We are grateful to the anonymous referees for their careful work and insightful comments on the paper.

\end{document}